\documentclass[12pt,letterpaper,reqno,oneside]{amsart}
\usepackage{amscd,amssymb}
\usepackage[arrow,matrix]{xy}
\usepackage{xcolor}
\usepackage[colorlinks=true,linkcolor=blue,citecolor=purple]{hyperref}
\usepackage{color}
\usepackage{graphicx}
\usepackage{fullpage}
\usepackage{mathrsfs}
\usepackage{tikz-cd}
\usetikzlibrary{calc, positioning, arrows.meta}
\usepackage{pifont}
\usepackage{mathtools}

\theoremstyle{plain}
\newtheorem{thm}{Theorem}[section]
\newtheorem{lem}[thm]{Lemma}
\newtheorem{prop}[thm]{Proposition}
\newtheorem{cor}[thm]{Corollary}

\theoremstyle{definition}
\newtheorem{rmk}[thm]{Remark}

\newtheorem{ex}[thm]{Example}

\newcommand{\CC}{\mathbb{C}}
\newcommand{\PP}{\mathbb{P}}
\newcommand{\ZZ}{\mathbb{Z}}

\newcommand{\QQ}{\mathbb{Q}}
\newcommand{\MM}{\mathbb{M}}

\newcommand{\Zc}{\mathcal{Z}}

\newcommand{\cd}{\mathrm{cd}}
\newcommand{\adj}{\mathrm{adj}}
\newcommand{\HF}{\mathrm{HF}}
\newcommand{\HP}{\mathrm{HP}}

\newcommand{\Syz}{\mathrm{Syz}}
\newcommand{\Proj}{\mathrm{Proj}}
\newcommand{\ann}{\mathrm{ann}}
\newcommand{\coker}{\mathrm{coker}\, }
\newcommand{\im}{\mathrm{Im}}
\newcommand{\rank}{\mathrm{rank}}

\newcommand{\depth}{\mathrm{depth}}
\newcommand{\Supp}{\mathrm{Supp}}

\newcommand{\Sc}{\mathcal{S}}

\newcommand{\Rc}{\mathcal{R}}

\newcommand{\Xr}{\mathscr{X}}
\newcommand{\Sr}{\mathscr{S}}

\newcommand{\mf}{\mathfrak{m}}
\newcommand{\pf}{\mathfrak{p}}
\newcommand{\qf}{\mathfrak{q}}
\newcommand{\Rf}{\mathfrak{R}}

\begin{document}

\title[]
{Implicitization of rational hypersurfaces by syzygies with respect to coefficient ideals}

\author{Falai Chen}
\address{School  of Mathematics, University of Science and Technology of China, Hefei, Anhui 230026,  China.}
\email{chenfl@ustc.edu.cn}

\author{Yunzhi Wang}
\address{School  of Mathematics, University of Science and Technology of China, Hefei, Anhui 230026,  China.}
\email{wangyunzhi@mail.ustc.edu.cn}

\subjclass{13D02,13P25,13D45,14Q10,14M25}
\keywords{Rational hypersurfaces, implicitization, toric variety, Rees algebra, coefficient ideal}

\date{\today}

\begin{abstract} 
We study rational hypersurfaces $\Sr$ defined as the closure of the image of a generically finite rational map $\phi:\mathscr{X}\rightarrow \PP^{n+1}$, where $\Xr$ is an $n$-dimensional toric variety. We provide matrix representations for the implicitization of $\Sr$ that are constructed from the coefficients of linear syzygies and quadratic syzygies of the parametric equations. A central feature of our construction is the restriction of all coefficients in the Cox ring $R$ to a specific coefficient ideal $J$. In the two-dimensional case, this approach eliminates the need for $\phi$ to be locally a complete intersection at the base points, that is, the determinant of the implicitization matrix is equal to a power of the implicit equation for arbitrary base points. This result generalizes several previous results in surface implicitization. 
\end{abstract}

\maketitle

\section{Introduction}\label{sec:intro}
The implicitization problem involves computing the implicit equation of a rational hypersurface from its parametric equations. The quest for a matrix-based solution dates back to the 18th century and the work of mathematicians such as B\'{e}zout, Cayley, Sylvester and Dixon, and continued into the early 20th century. Over the past few decades, this problem has regained significant attention due to its applications in geometric modeling. The 1980s and 1990s saw the advent of Gröbner basis methods and substantial progress in resultant-based techniques. However, Gröbner basis computations are highly expensive and resultant methods fail for surfaces with base points.

To overcome the limitations mentioned above, Thomas Sederberg and one of the authors of this paper introduced the method of moving surfaces \cite{SC95}. This approach constructs a matrix whose elements are composed of higher degree polynomials along with linear polynomials. Experimental results demonstrate its effectiveness for rational parametric surfaces, even those with complicated base points.

\medskip
In this article, we study the implicitization problem for rational hypersurfaces defined by a toric parameterization. Let $\Xr$ be an $n$-dimensional $(n\geqslant2)$ non-degenerate simplicial toric variety over an algebraically closed field $k$. The cases of the most interest are triangular surfaces and tensor product surfaces, where $\Xr=\PP^2,\ \PP^1\times\PP^1$ respectively. Let $R$ be the Cox ring with the class group $G$ and the irrelevant ideal $B$. Let $S=k[x_0,\cdots,x_{n+1}]$ be the homogeneous coordinate ring of $\PP^{n+1}$. We consider the following rational map defined by a homogeneous ideal $I=(f_0,\dots,f_{n+1})$ in $R$ :  
$$\phi: \ \   \Xr \ \  \longrightarrow \ \ \ \ \PP^{n+1}$$
$$\ \ \ \ \ \ \ \ \ \ \ \ \ \  p\longmapsto (f_0,\cdots,f_{n+1})(p)$$
where $f_0,\cdots,f_{n+1}$ are homogeneous polynomials of the same degree $\gamma\in G$. We assume $\phi$ is generically finite, with its base locus $V(I)$ supported on a finite set of points, and $\Xr$ is smooth at $V(I)$. This rational map defines a hypersurface $\mathscr{S}$ that is the closure of the image. Let $H$ denote its implicit equation. 

\medskip
The method of moving surface concerns the following polynomial
\begin{equation}\label{eq:defMS}
F=\sum_{i_0+\cdots+i_{n+1}=N}g_{i_0\cdots i_{n+1}}x_0^{i_0}\cdots x_{n+1}^{i_{n+1}}
\end{equation}
such that $F(f_0,\cdots,f_{n+1})\equiv 0$, where $g_{i_0\cdots i_{n+1}}$ are homogeneous polynomials in $R$ of a common degree. These coefficients correspond to the $N$-th degree syzygies of $I$, which form the kernel of the following map induced by $I^N$:
\begin{equation}\label{eq:defQk}
0\rightarrow \Syz(I^N;R)\rightarrow R^{\binom{N+n+1}{n+1}}  \xrightarrow{I^N}  R(N\gamma). \nonumber
\end{equation}
We denote the set of all $F$ by $Q_N(I;R)$ to emphasize the grading. In \cite{SC95}, linear (resp.\ quadratic) syzygies are called moving planes (resp.\ quadrics) in dimension two. It is easy to see that $Q=\sum_{i\geqslant0} Q_i$ is a $G \oplus \ZZ$-graded $R\otimes_k S$-module, since multiplying an $N$-th degree syzygy by a polynomial in $S$ yields a syzygy of higher degree. 

For $\mu\in G$, we can recursively select a set of minimal $S$-generators $L_1,\cdots,L_l$ for $Q_{\mu}$: having chosen minimal generators $L_1,\cdots,L_{l_N}$ of $\sum_{k\leqslant N}Q_{\mu,k}$, we add $L_{l_N+1},\cdots,L_{l_{N+1}}$ as the minimal generators of $Q_{\mu,N+1}/\langle L_1,\cdots,L_{l_N} \rangle$. Then, given a $k$-basis $b_1,
\cdots,b_q$ of $R_{\mu}$, the \emph{implicitization matrix} is defined as the following coefficient matrix $\MM_{\mu}$:
\begin{equation}\label{eq:defM}
(b_1\ \cdots\ b_q)\ \MM_{\mu}=(L_1\ \cdots\ L_l  ).
\end{equation}
It is a column-homogeneous matrix over $S$ with each column of degree $N$ corresponding to an $N$-th degree syzygy. If we choose $L_1,\cdots,L_l$ as minimal $S$-generators of $Q_{\mu,1}$ (resp. $Q_{\mu,2}$), then the matrix satisfying $\eqref{eq:defM}$ is called a \emph{linear matrix} (resp. \emph{quadratic matrix}). Note that quadratic matrices may consist of linear and quadratic columns.

\medskip
In \cite{SC95}, it is experimentally shown that for some special $
\mu$, the quadratic matrix $\MM_{\mu}$ is often square with the determinant equal to the implicit equation $H$.  The first rigorous proof of the two-dimensional case is given in \cite{CGZ00}, under the assumptions that the rational map has no base points and no linear syzygies of $\mathbf{f'}$, where $\mathbf{f}'=(f_0',\cdots,f_n')$ are $n+1$ generic $k$-linear combinations of $\mathbf{f}$. Several years later, \cite{BCD03,AHW05} respectively prove the results for triangular and tensor product surfaces that are locally complete intersections at the base points (abbreviated as l.c.i.) with additional constraints on base points. In \cite{BJ03}, moving surface modules are first interpreted as Rees algebras and the determinants are computed by approximation complexes. This provides a more algebraic framework, leading to subsequent research including \cite{BC05,BC21,BCS10,BD16,Bot11}. The case involving linear syzygies of $\mathbf{f'}$ has recently been considered in \cite{BC21,LCS19,LCSG22}. Research work on non-l.c.i.\ surfaces can be found in \cite{BD16, HW23, ZSCC03}.  

Let $\mathscr{S}$ be a tensor product surface of bi-degree $\gamma=(\gamma_1,\gamma_2)$ with at most l.c.i.\ base points. In \cite{Bot11}, it is shown that the greatest common divisor of the maximal minors of $\MM_{2\gamma_1-1,\gamma_2-1}$ equals $H^{\deg(\phi)}$. Subsequently, \cite{BC21} proposes that the quadratic matrix $\MM_{\mu,\gamma_2-1}$ also yields the implicit equation for some $\mu<2\gamma_1-1$. These two matrices become non-square when $I$ possesses certain linear syzygies of $\mathbf{f'}$. Furthermore, if $I$ has a.l.c.i. base points ($I_{\pf}$ is locally generated by $n+1$ elements) then extra factors appear in the determinants. In the presence of non-a.l.c.i.\ base points, these two matrices may even fail to be of full rank.

\medskip
Our work focuses on three goals. First, we seek to provide a matrix representation that does not require the l.c.i.\ assumption at least in the two-dimensional case. Second, we aim to extend previous results to toric-parametrized hypersurfaces without restriction on dimension. Third, when the implicitization matrix is non-square, we aim to derive a more compact representation by reducing the number of rows, columns, and the differences between them.

Following the idea of \cite{ZSCC03}, we construct the implicitization matrix $\MM_{\mu}$ using syzygies with coefficients in a homogeneous ideal $J$. This means we only use syzygies with $g_{i_0\cdots i_{3}}\in J$ in \eqref{eq:defMS} and choose $b_1,\cdots,b_q$ as a basis of $J_{\mu}$. Consequently, our matrices have fewer rows than those using all bases of $R_{\mu}$. We impose the following conditions on the \emph{coefficient ideal} $J$. 
\par \ding{172} $I$ is locally generated by $n$ elements with respect to $J$ at the base points, that is, for any $\mathfrak{p}\in V(I)$, there exist $f_1'\cdots,f_n'\in I_{\pf}$ such that $I_{\pf}J_{\pf}=(f_1',\cdots,f_n')J_{\pf}$.
    
\par \ding{173} $I\subset J$.
\par \ding{174} $J=J^{sat}$, where $J^{sat}=J:B^{\infty}$ is the \emph{saturation} of $J$ with respect to $B$.
\par  \ding{175} $I$ is locally generated by a proper sequence with respect to $J$ at the base points.
\par \ding{176} $I$ is locally of linear type with respect to $J$ at the base points.

We will prove that conditions \ding{172} and \ding{173} imply \ding{175} and  \ding{176}. Recall that $I$ is a \emph{locally complete intersection} at the base point if the number of minimal generators of $I_{\pf}$ equals $\depth(I_{\pf})$, which is $n$. Thus, condition \ding{172} is a natural generalization of the l.c.i.\ condition, focusing on the number of generators rather than $\depth(I_{\pf},J_{\pf})$. 

\medskip
The main result of the current paper is the following:
\begin{thm}\label{thm:det}
\leavevmode
    \begin{itemize}
        \item[i)] If $I$ admits a coefficient ideal, then there exists $\mu\in G$ and the quadratic matrix $\MM_{\mu}$, such that the greatest common divisor of its maximal minors equals $H^{\deg(\phi)}$.
        \item[ii)] Coefficient ideals exist for any $I$ in dimension two. This yields quadratic implicitization matrices for rational surfaces with arbitrary base points.
        \item[iii)] Let $f_0',\cdots,f_n'$ be $n+1$ generic $k$-linear combinations of $f_0,\cdots,f_{n+1}$. For the $\mu$ in $\mathrm{i)}$, if $R_{\mu-\gamma}=0$ and $\Syz(\mathbf{f'};J)_{\mu}=0$, then $\MM_{\mu}$ is a square matrix. 
	\end{itemize}	
\end{thm}
When $J=R$, our method recovers the main results in \cite{BC21,Bot11} as a special case.

\medskip
The article is organized as follows. Section \ref{sec:pre} provides basic introductions to Rees algebras, determinants and approximation complexes. In Section \ref{sec:quadratic}, we prove our main theorems on quadratic matrices and linear matrices in any dimension $n\geqslant2$.  In Section \ref{sec:adj}, we study conditions \ding{172}-\ding{176} locally at the base points and propose several methods for constructing the local coefficient ideals in dimension two. In Section \ref{sec:size}, we analyze the size of the matrix. In Section \ref{sec:ex}, we illustrate our method by several examples.

\section{Preliminaries}\label{sec:pre}

\subsection{Toric varieties} We refer to \cite{CLS11} for a comprehensive introduction to toric varieties.
Let $\Sigma$ be the fan associated with $\Xr$, which is in the lattice $\ZZ^n$. Denote by $\Sigma(1)=\{\rho_1,\cdots,\rho_l\}$ the set of $1$-dimensional cones in $\Sigma$ and by $u_i$ the primitive generator of $\rho_i\cap \ZZ^n$. The \emph{Cox ring} $R$ is the polynomial ring $k[x_1,\cdots,x_l]$ with the $G$-grading $\deg x_i=\varphi_2(D_{\rho_i})$ induced by the following exact sequence:
$$0\rightarrow \ZZ^n \xrightarrow{\varphi_1} \bigoplus_{1\leqslant i \leqslant l} \ZZ D_{\rho_i} \xrightarrow{\varphi_2} G \rightarrow 0 $$
where $D_{\rho_i}$ is the divisor corresponding to $\rho_i$ and $\varphi_1(m)=\sum_i \langle m,u_i \rangle D_{\rho_i}$. For each $\sigma\in \Sigma$, define $x^{\hat{\sigma}}=\prod_{\rho\notin \sigma(1)}x_{\rho}$. The \emph{irrelevant ideal} is $B=(x^{\hat{\sigma}}\mid\sigma\in \Sigma(n))$, and it satisfies $\depth(B)\geqslant2\ $ by \cite[Lemma 1.4]{Cox95}.

For simplicial toric varieties, there hold the toric Nullstellensatz and ideal-variety correspondence similar to the projective space case. This enables us to study the coefficient ideal $J$ locally at each base point via primary decomposition. 

For a finitely generated graded module $M$, there is an induced quasi-coherent sheaf $\widetilde{M}$. This yields a $\Proj$ construction for the $R$-algebras $R/I,\ R/J$ and subalgebras of $R[t]$, which we denote by $\Proj_{\Xr}$.
Write $M(\mu)$ for the twisted module with $M(\mu)_{\sigma}=M_{\mu+\sigma}$. There exists the following exact sequence (\cite[Theorem 9.5.7]{CLS11}), which will be used frequently: 
\begin{equation}\label{eq:H01}
0\rightarrow H_B^0(M)\rightarrow M \rightarrow \bigoplus_{\mu\in G} H^0(\Xr,\widetilde{M(\mu)}) \rightarrow H_B^1(M) \rightarrow 0 .
\end{equation}
Moreover, the following isomorphism holds for $i\geqslant2$:
\begin{equation}\label{eq:H2}
H_B^{i}(M)=\bigoplus_{\mu\in G} H^{i-1}(\Xr,\widetilde{M(\mu)}).
\end{equation}
Let $\HF(M,\mu), \HP(M,\mu)$ be the Hilbert function and the Hilbert polynomial, respectively. We have the Grothendieck-Serre formula, which relates these two numbers:
$$\HF(M,\mu)-\HP(M,\mu)=\sum_{i\geqslant0}(-1)^{i} HF(H_B^i(M),\mu).$$

 Then we come back to the rational map $\phi:
 \Xr\rightarrow\PP^{n+1}$. The degree of the implicit equation $H$ can be calculated using intersection theory. In fact, we have 
\begin{equation}\label{eq:DEG1}
\deg(\phi)\cdot \deg(H)= (\mathcal{O}_{\Xr}(\gamma))^n-\sum_{\mathfrak{p}\in V(I)}e(I_{\mathfrak{p}})
\end{equation}
by \cite[Proposition 4.4]{Ful98} and \cite[Appendix]{Cox01}. The intersection number $(\mathcal{O}_{\Xr}(\gamma))^n$ can be calculated by \cite[Theorem 12.5.3]{CLS11}.

\subsection{Rees modules} 
Recall that the \emph{Rees algebra} $\Rc_I(R)=\bigoplus_{k\geqslant0} I^kt^k\subset R[t]$ is a $G\oplus \ZZ$-graded algebra with $\deg(t)=(0,1)$. Set $T=R\otimes_kS$. There is a surjection:
$$T\rightarrow \Rc_I(R)$$
$$x_i\longmapsto f_it$$
and its kernel is generated by all polynomials $F\in T$ such that $F(f_0,\cdots,f_{n+1})\equiv 0$, which is the module $Q(I;R)$ defined in Section \ref{sec:intro}. 

$\Proj_{\Xr}\Rc_I(R)$, which is a closed subscheme of $\subset \Xr\times_k \PP^{n+1}$, carries significant geometric content. Set $U=\Xr\backslash V(I)$. Then $\Proj_{\Xr}\Rc_I(R)$ is the Zariski closure of the graph of the morphism $\phi|_U:U\rightarrow \PP^{n+1}$. Let $\pi_1,\pi_2$ be the natural projections: 
$$\Xr\xleftarrow{\pi_1}\Proj_{\Xr}\Rc_I(R)\xrightarrow{\pi_2}\PP^{n+1}.$$
One has $\deg(\pi_1)=1$ and $\pi_2(\Proj_{\Xr}\Rc_I(R))=\mathscr{S}$. 

Define the \emph{Rees module} of $I$ with respect to an ideal $J$ as $\Rc_I(J)=\bigoplus_{k\geqslant0} JI^kt^k$. It is a $G\oplus\ZZ$-graded module and can be viewed as the extension of the ideal $J$ to the Rees algebra. Moreover, there is a similar surjection $J[x_0,\cdots,x_{n+1}]\rightarrow \Rc_I(J)$ mapping $x_i$ to $f_it$, whose kernel is $Q(I;J)=J[x_0,\cdots,x_{n+1}]\cap Q(I;R)$, i.e, the elements of $Q$ with coefficients in $J$.

Let $Q_k$ be the graded component of degree $k$ with respect to $x_0,\cdots,x_{n+1}$. Define $\Sc_I(J)=J[x_0,\cdots,x_{n+1}]/(J\cap Q_1)$ as the \emph{symmetric module} of $I$ with respect to $J$, where $J$ denotes the extension ideal of $J$ in $R[x_0,\cdots,x_{n+1}]$ by abuse of notation. We have a canonical surjective map $\alpha:\Sc_I(J)\rightarrow \Rc_I(J)$. Let $\langle Q_k\rangle$ be the $S$-module generated by $Q_k$. We say $I$ is of linear type with respect to $J$ if $\alpha$ is an isomorphism, which is equivalent to $J\cap \langle Q_1\rangle=J\cap Q$. 

One can see that the implicitization matrix $\MM_{\mu}$ introduced in \eqref{eq:defM} is exactly a minimal presentation matrix of $\Rc_I(R)_{\mu}$ for some $\mu\in G$ and the matrix we will focus on hereafter is the minimal presentation matrix of $\Rc_I(J)_{\mu}$. Moreover, if we consider only the linear part, it corresponds to the minimal presentation matrix of $\Sc_I(J)_{\mu}$. 

\subsection{Determinants} The original idea of the method of moving surfaces is to find a square matrix $\MM$ over $S$, such that the implicit equation can be written as $H=\det(\MM)$, c.f. \cite{CGZ00,SC95}. However, some conditions on linear syzygies of $\mathbf{f'}$ are needed by the proofs of \cite[Theorem 4.4]{CGZ00} and \cite[Proposition 5.1]{BC21}. Hence, we should also consider non-square matrices.

We use the definition of \emph{determinants} for complexes in \cite[Appendix A]{GKZ08}. For a $S$-module $M$, we set $\det(M)=\det(F_{\bullet})$, where $F_{\bullet}$ is a finite free resolution of $M$. In addition, we set $\det(\MM)=\det(\coker(\MM))$ for a matrix $\MM$ over $S$. This determinant is equal to the greatest common divisor of all maximal minors of $\MM$ by \cite[Appendix A, Theorem 34]{GKZ08}. As noted, free resolutions of length one are of particular interest because they yield square matrix representations. Moreover, free resolutions of length two also have nice properties, as their determinants can be expressed as ratios of determinants of two submatrices, c.f. \cite[Appendix A, Propsition 11]{GKZ08}.

\subsection{Approximation complexes} Under certain hypotheses, free resolutions of symmetric modules can be constructed using approximation complexes. This motivates the use of symmetric modules as approximations of Rees modules. We refer to \cite{HSV83,Vas94} for detailed discussions on approximation complexes.

Set $K_{\bullet}(\mathbf{f};J)$ for the graded Koszul complex of $f_0,\cdots,f_{n+1}$ with respect to $J$. Thus $K_p(\mathbf{f};J)=\bigwedge^pJ(-i\gamma)^{n+2}$. Write $Z_p,B_p,H_p$ for the $p$-th modules of Koszul cycles, boundaries, and homologies, respectively. We observe that 
\begin{equation}\label{eq:Koszul}
Z_p(\mathbf{f};J)=J\cap Z_p(\mathbf{f};R),\ B_p(\mathbf{f};J)=J\cdot B_p(\mathbf{f};R).
\end{equation}

Denote by $\Zc_{\bullet}$ the graded \emph{approximation complex} of $f_0,\cdots,f_{n+1}$ with respect to $J$. We have 
$$\Zc_p=Z_p(p\gamma)\otimes S(-p).$$ 
The boundary maps $\partial_i$ of $\Zc_{\bullet}$ are induced by the boundary maps of $K_{\bullet}(\mathbf{x};J)$. Moreover, each entry of the matrix of $(\partial_i)_{\mu}$ for some $\mu\in G$ is a linear form in $S$ by the construction. Write $\mathcal{H}_i$ for the $i$-th homology module of $\Zc_{\bullet}$. A key property is that $\mathcal{H}_0=\Sc_I(J)$, since $\Zc_1$ is the $S$-module generated by all linear syzygies. Hence, the matrix of $(\partial_1)_{\mu}$ is precisely the linear matrix of interest.

We say that $g_1,\cdots,g_m\in R$ form a \emph{proper sequence} with respect to $J$ if 
$$g_{i+1}H_1(g_1,\cdots,g_{i};J)=0$$
for $1\leqslant i\leqslant m-1 $. The complex $\Zc_{\bullet}$ is acyclic if and only if $I=(\mathbf{f})$ can be generated by a proper sequence with respect to $J$.

\section{Quadratic matrices and linear matrices}\label{sec:quadratic}
Hereafter, all notations without an explicit coefficient module are understood to be with respect to a homogeneous ideal $J$, which is assumed to satisfy conditions \ding{173},\ding{174} throughout this section. In particular, some lemmas proved below also hold for $J=R$.

This section studies implicitization via quadratic matrices and linear matrices, i.e., the matrix $\MM_{\mu}$ whose columns 
are either a mixture of linear and quadratic forms, or purely linear. The answers will be shown to have a close relationship with the complements of the following two regions:
$$\mathscr{A}=\left(\bigcup_{2\leqslant p\leqslant n}(\Supp_G H_B^p(R)+(p-1)\gamma)\right)\cup\left(\bigcup_{p\geqslant2}\left(\Supp_GH_B^p(R)+(p-2)\gamma\right)\right),$$
\begin{equation}\label{eq:defR}
\mathfrak{R}=\bigcup_{p\geqslant2}\left(\Supp_GH_B^p(R)+(p-1)\gamma\right), \nonumber
\end{equation}
where for a $G$-graded module $M$, the graded support is defined as $\Supp_GM=\{\mu\mid M_{\mu}\neq 0 \}$.

\begin{ex} \label{ex:region}
If $\Xr=\PP^n$, then
\begin{eqnarray}
& \Rf^c=n\gamma-n+\ZZ_{\geqslant0}, \nonumber \\
& \mathscr{A}^c=(n-1)\gamma-n+\ZZ_{\geqslant0}. \nonumber
\end{eqnarray} 
If $\Xr=(\PP^1)^n$ with $\gamma=(\gamma_1,\cdots,\gamma_n)$ and $\gamma_1\geqslant \cdots\geqslant \gamma_n$, then
\begin{eqnarray}
& (n\gamma_1-1,(n-1)\gamma_2-1,\cdots,\gamma_n-1)+\ZZ_{\geqslant0}^n\subset \Rf^c,  \nonumber\\ 
& \ZZ\times ((n-1)\gamma_2-1,\cdots,\gamma_n-1)\subset \mathscr{A}^c. \nonumber
\end{eqnarray} 
See \cite[Lemma 6.7]{Bot11} for more details when $\Xr$ is a multi-projective space.
\end{ex}

 \subsection{Quadratic matrix representations}
Assuming \ding{176}, localizing at each base point yields $\Sc_I/H^0_B(\Sc_I)=\Rc_I$. Thus, all elements of $Q_N$ for $N\geqslant2$ arise from $H_B^0(\Sc_I)$. This leads to the question: for which degree $\mu$ is the module $H^0_{B}(\mathcal{S}_I)_{\mu}$ generated by quadratic syzygies? We resolve it using the spectral sequence associated with the Čech complex of $Z_{\bullet}$. To do so, we first analyze the local cohomology of Koszul cycles.

\begin{lem}\label{lem:HH}
The following properties hold:
	\begin{itemize}
		\item[i)] $Z_{n+2}=H_{n+2}=0$.
		\item[ii)] $H_p(\mathbf{f};R)=0,$ $Z_p(\mathbf{f};J)=B_p(\mathbf{f};R)=Z_p(\mathbf{f};R)$ for $p\geqslant3$. 
        \item[iii)] $H_B^i(H_p)=0$ for $i\geqslant2$ and any $p$.
	\end{itemize}	
\end{lem}
\begin{proof}
We have $\depth(I,J)\geqslant 1$ since $R$ is an integral domain. This indicates $Z_{n+2}=H_{n+2}=0$. Since $\depth(I,R)=n$, we obtain $H_p(\mathbf{f};R)=0$ for $p\geqslant3$. Thus $Z_p(\mathbf{f};J)=J\cap Z_p(\mathbf{f};R)=J\cap B_p(\mathbf{f};R)$. Each entry in the matrix of $B_p(\mathbf{f};R)$ lies in $J$ since $I\subset J$. Hence, $Z_p(\mathbf{f};J)=B_p(\mathbf{f};R)=Z_p(\mathbf{f};R)$ for $p\geqslant3$. Furthermore, we have $\Supp_{\Xr} H_p(\mathbf{f};J)\subset\Supp_{\Xr}(H_p(\mathbf{f};R))\bigcup V(J)\subset V(I)$ since $I\subset J$. For any $i\geqslant2$, this indicates $$H_B^i(H_p)=\bigoplus_{\mu\in G} H^{i-1}(\Xr,\widetilde{H_p(\mu)})=\bigoplus_{\mu\in G} H^{i-1}(V(I),\widetilde{H_p(\mu)})=0$$  by \eqref{eq:H2} and $i-1>\dim V(I)$. 
\end{proof}

\begin{lem}\label{lem:HJ}
The following properties hold:
	\begin{itemize}
		\item[i)] $H_B^i(J)=0$ for $i=0,1$.
		\item[ii)] $H_B^2(J)_{\mu}=H_B^1(R/J)_{\mu}$, if $H^2_B(R)_{\mu}=0.$
        \item[iii)] $H_B^i(J)_{\mu}=0$ for $i\geqslant3$, if $H_B^i(R)_{\mu}=0$.
        \item[iv)] $H_B^i(J)_{\mu}=0$ for $i\geqslant2$, if $H_B^{i+j-1}(R)_{\mu-j\gamma}=0$ for any $j\geqslant1$.
        \item[v)] $\mathscr{A}=\left(\bigcup_{2\leqslant p\leqslant n}(\Supp_G H_B^p(R)+(p-1)\gamma)\right)\cup\bigcup_{p\geqslant2}\left(\bigcup_{0\leqslant j\leqslant p-2} \left(\Supp_GH_B^p(R)+j\gamma\right)\right).$
	\end{itemize}	
\end{lem}
\begin{proof}
We deduce $H^i_B(R/J)=0$ for $i\geqslant2$ from \eqref{eq:H2}. Furthermore, note that $H_B^1(R)=0$ because $\depth(B)\geqslant2$, and $H_B^0(R/J)=0$ by condition \ding{174}. Then i), ii), iii) follow from the long exact sequence induced by $0 \rightarrow J  \rightarrow R  $:
$$\xymatrix{H^{i-1}_{B}(R)\ar[r] & H^{i-1}_{B}(R/J) \ar[r] & H^{i}_{B}(J) \ar[r] & H^{i}_{B}(R).   }    $$
Consider the spectral sequence for $E_0^{p,q}= \mathcal{C}_B^{q}(K_{p})$, where $\mathcal{C}_B^{\bullet}(M)$ is the Čech complex of a module $M$. We always write it as a second quadrant cohomology spectral sequence. The second page of the row-filtered spectral sequence $'E_{2}$ is: 
$$\begin{array}{cccccc}
\vdots & \vdots &  & \vdots & \vdots \\
0 & 0  & \cdots &0  & 0\\
0 & H^1_{B}(H_{n+1}) & \cdots & H^1_{B}(H_1)&H^0_{B}(H_0)  \\
0 & H^0_{B}(H_{n+1}) &\cdots & H^0_{B}(H_1)&H^0_{B}(H_0)
\end{array}$$
We deduce $'E_{\infty}^{j,i+j}=0$ for $i\geqslant2$ and any $j$. The first page of the column-filtered spectral sequence $''E_{1}$:
$$\begin{array}{cccccc}
\vdots & \vdots &  & \vdots & \vdots \\
 H^3_{B}(K_{n+2}) & H^3_{B}(K_{n+1}) &\cdots & H^3_{B}(K_1)& H^3_{B}(K_0)  \\
 H^2_{B}(K_{n+2}) & H^2_{B}(K_{n+1}) &\cdots & H^2_{B}(K_1)& H^2_{B}(K_0)  \\
0 & 0  & \cdots &0  & 0\\
0 & 0  & \cdots &0  & 0
\end{array}$$
Write $''d_j^{p,q}$ for the differential map from $''E_j^{p,q}$. Under the assumption of iv), the canonical map $(''d_1^{1,i})_{\mu}:H_B^i(K_1)_{\mu}\rightarrow H_B^i(K_0)_{\mu}$ is zero for $i\geqslant3$ by iii). Moreover, $H_B^2(K_1)_{\mu}$ is annihilated by $J$ from ii). Thus $(d_1^{1,2})_{\mu}=0$ follows from $I\subset J$ and the construction of Koszul maps. This means $(E_3^{0,i})_{\mu}=(E_2^{0,i})_{\mu}$. Furthermore, since $H_B^{i+j-1}(K_{j})_{\mu}=0$ for $j\geqslant 1$ by definition and $(''E_j^{j,i+j-1})_{\mu}$ is a subquotient of $H_B^{i+j-1}(K_{j})_{\mu}$, we can see $(d_{j}^{j,i+j-1})_{\mu}=0$. This means $(''E_{\infty}^{0,i})_{\mu}=(''E_{2}^{0,i})_{\mu}$. Finally, by comparing two filtered spectral sequences, we deduce $H^i_B(J)_{\mu}=H^i_B(K_0)_{\mu}=(''E_{\infty}^{0,i})_{\mu}=0$.
\par Note that iv) holds for $J=R$. Then v) follows by using iv) recursively.
\end{proof}

Now we come to compute the local cohomologies of Koszul cycles. 
\begin{lem}\label{lem:HZ}
The following properties hold:
	\begin{itemize}
		\item[i)] $H_B^0(Z_p)=H_B^1(Z_p)=0$ for any $p$.
		\item[ii)] $H_B^i(Z_p)=H_B^i(B_p)$ for any $p$ and $i\geqslant3$.
        \item[iii)] $H^i_B(Z_p)_{\mu}=0$ for $i\geqslant3$, if $H_B^{i+j-1}(K_{p+j})_{\mu}=0$ for any $j\geqslant 1$. 
        \item[iv)]  $H_B^2(Z_p)_{\mu}=H_B^1(B_{p-1})_{\mu}=H_B^0(H_{p-1})_{\mu}$, if $H_B^2(K_p)_{\mu}=0$.
	\end{itemize}	
\end{lem}
\begin{proof}
We have $H_B^0(Z_p)=H_B^0(B_p)=0$ for any $p$, since $R$ is an integral domain. From $\depth(B)\geqslant2$, we obtain $H_B^1(K_p)=0$. Then $H_B^1(Z_p)=0$ follows from the long exact sequence induced by $B_{p-1}=K_p/Z_p$. This proves i). As $H_B^i(H_p)=0$ for $i\geqslant2$, the statement ii) follows from the long exact sequence induced by $H_p=Z_p/B_p$.
\par Under the assumption of iii), we have the embedding $H_B^{i+j-1}(B_{p+j-1})_{\mu}\hookrightarrow H_B^{i+j}(Z_{p+j})_{\mu}$ from the long exact sequence induced by $B_{p+j-1}=K_{p+j}/Z_{p+j}$. This yields the embedding 
$$H_B^{i+j-1}(Z_{p+j-1})_{\mu}\hookrightarrow H_B^{i+j}(Z_{p+j})_{\mu}$$
by ii). Then we obtain the embedding $H^i_B(Z_p)_{\mu}\hookrightarrow H^{n+i+2-p}_B(Z_{n+2})_{\mu}=0$ by composing all these maps. This indicates $H^i_B(Z_p)_{\mu}=0$.
\par Under the assumption of iv), the long exact sequence induced by $B_{p-1}=K_{p}/Z_p$ yields $H_B^2(Z_p)_{\mu}=H_B^1(B_{p-1})_{\mu}$. Moreover, combining i) and the long exact sequence induced by $H_{p-1}=Z_{p-1}/B_{p-1}$ gives $H_B^1(B_{p-1})_{\mu}=H_B^0(H_{p-1})_{\mu}$. Finally, we obtain $H_B^2(Z_p)_{\mu}=H_B^0(H_{p-1})_{\mu}$. 
\end{proof}

Our next task is to compute the local cohomologies of the blowup modules $\Sc_I$ and $\Rc_I$. Let $\delta^j_i: H^j_{B}(\Zc_i)\rightarrow H^j_{B}(\Zc_{i-1})$ be the canonical maps induced by local cohomology. 

\begin{lem}\label{lem:A}
Let $\mu\notin \mathscr{A}$. The following properties hold:
	\begin{itemize}
		\item[i)] $H_B^1(\Sc_I)_{\mu}=\ker (\delta^2_1)_{\mu}/\im(\delta^2_2)_{\mu}, H_B^2(\Sc_I)_{\mu}=\coker (\delta^2_1)_{\mu}$ and $H_B^i(\Sc_I)_{\mu}=0$ for $i\geqslant3$. 
		\item[ii)] If $J$ satisfies \ding{175}, then 
        $(\mathcal{H}_p)_{\mu}=0$ for $p\geqslant1$ and $H_B^0(\Sc_I)_{\mu}=\ker(\delta^2_2)_{\mu}$.
	\end{itemize}	
\end{lem}
\begin{proof}
From Lemma \ref{lem:HJ} v), we can see $H_B^i(R)_{\mu-j\gamma}=0$ for $\mu\notin\mathscr{A}$ and $i\geqslant2,j\leqslant i-2$.
Using Lemma \ref{lem:HJ} iii) and Lemma \ref{lem:HZ} iii), this yields $H_B^i(K_p)_{\mu+p\cdot \gamma}=0$ for $i\geqslant3$ and any $p$. Thus, we deduce $H_B^i(Z_p)_{\mu+p\cdot \gamma}=0$ for $i\geqslant3$ and any $p$, which implies $H_B^i(\Zc_p)_{\mu}=0$. 

Next, we analyze $H_B^i(\mathcal{Z}_p)$ for $i=2$ and $p\geqslant3$. We have $Z_i(\mathbf{f};J)=B_i(\mathbf{f};R)=Z_i(\mathbf{f};R)$ from Lemma \ref{lem:HH} ii). If $j\geqslant n-1$, since $p+j+1>n+2$, we obviously have $H_B^{2+j}(K_{p+j+1}(I;R))_{\mu+p\gamma}=0$. Meanwhile, for $j\leqslant n-2$, the definition of $\mathscr{A}$ implies $H_B^{2+j}(K_{p+j+1}(I;R))_{\mu+p\gamma}=0$. Applying Lemma \ref{lem:HZ} iii), we deduce
$$H_B^3(Z_{p+1}(\textbf{f};R))_{\mu+p\gamma}=0.$$
From the long exact sequences of $B_p=K_{p+1}/Z_{p+1}$ and $H_p=Z_p/B_p$, we have the following two surjections:
$$H_B^2(K_{p+1}(\mathbf{f};R))_{\mu+p\gamma}\twoheadrightarrow H_B^2(B_p(\mathbf{f};R))_{\mu+p\gamma}\twoheadrightarrow H_B^2(Z_p(\mathbf{f};R))_{\mu+p\gamma}.$$
Moreover, we have $H_B^2(K_{p+1}(\mathbf{f};R))_{\mu+p\gamma}=0$ from $\mu\notin\mathscr{A}$, which implies $H_B^2(Z_p(\mathbf{f};R))_{\mu+p\gamma}=0$. Finally, we obtain $H_B^i(\mathcal{Z}_p)_{\mu}=0$ for $i=2$ and $p\geqslant3$.

Finally, we compute the spectral sequence for $\mathcal{C}_{B}^{\bullet}\mathcal{Z}_{\bullet}$. The second page of the row-filtered spectral sequence is:
$$\begin{array}{cccccc}
\vdots & \vdots &  & \vdots & \vdots \\
0 & 0 & \cdots & 0 & H^3_B(\mathcal{S}_I)  \\
0 & 0 & \cdots & 0 & H^2_B(\mathcal{S}_I)  \\
0 & H^1_{B}(\mathcal{H}_{n+1}) & \cdots & H^1_{B}(\mathcal{H}_1) & H^1_B(\mathcal{S}_I)  \\
0 & H^0_{B}(\mathcal{H}_{n+1}) & \cdots & H^0_{B}(\mathcal{H}_1) & H^0_B(\mathcal{S}_I)  \\
\end{array}$$
The first page of the column-filtered spectral sequence in degree $\mu$ is:
	\[
	\xymatrixrowsep{0.3em}
	\xymatrixcolsep{1.4em}
	\xymatrix{
	 0 & 0 & 0  & 0 \\
	   0 & H^2_B(\Zc_2)_{\mu} \ar[r]^{(\delta_2^2)_{\mu}} & H^2_B(\Zc_1)_{\mu} \ar[r]^{(\delta_1^2)_{\mu}} &  H^2_B(\Zc_0)_{\mu} \\
    0 & 0 & 0  & 0 \\
     0 & 0 & 0  & 0 \\
	}
	\]
 If $J$ satisfies condition  \ding{175}, then $\mathcal{H}_i$ is $B$-torsion for $i>0$. Consequently, $H^0_B(\mathcal{H}_p)=\mathcal{H}_p$ and $H^1_B(\mathcal{H}_p)=0$ for any $p>0$. Statements i) and ii) then follow from the spectral sequence.
\end{proof}

\begin{lem}\label{lem:Rees}
Let $\mu\notin \mathscr{A}$. The following properties hold:
 	\begin{itemize}
        \item[i)] $H_B^i(\Rc_I)_{\mu}=H_B^i(\Sc_I)_{\mu}$ for $i\geqslant2$.
        \item[ii)] $H_B^1(\Rc_I)_{\mu}=0$, if $H_B^1(\Sc_I)_{\mu}=0$.
        \item[iii)] If $H_B^0(H_0(\mathbf{f};R))_{\mu+\gamma}=0$, then  $H_B^i(\Rc_I(R))_{\mu}=0$ for $i\geqslant0$.

	\end{itemize}	
\end{lem}
\begin{proof}
Consider the following exact sequence:
\begin{equation}\label{eq:RISI}
0\longrightarrow  Q(I;J)/Q_1(I;J)\longrightarrow \Sc_I(J) \longrightarrow \Rc_I(J)\longrightarrow 0.
\end{equation}	
We have
 \begin{align*}
\Supp_{\Xr} (Q(I;J)/Q_1(I;J))& =\Supp_{\Xr} (J\cap Q(I;R))/(J\cap Q_1(I;R))\\
& \subset \Supp_{\Xr}(Q(I;R)/Q_1(I;R)).
\end{align*}
Thus, $Q(I;J)/Q_1(I,J)$ is supported on non-l.c.i.\ base points. Hence, $H_B^i(Q(I;J)/Q_1(I;J))=0$ for $i\geqslant2$. Then the long exact sequence of \eqref{eq:RISI} yields i), ii). 

 Now, we prove iii). We have $H_B^2(\Zc_0(\mathbf{f};R))_{\mu}=H_B^2(R)_{\mu}\otimes S=0$ by the definition of $\mathscr{A}$. Moreover, $H_B^2(\Zc_1(\mathbf{f};R))_{\mu}=H_B^2(Z_1(\mathbf{f};R))_{\mu+\gamma}\otimes S(-1)=0$ follows from Lemma \ref{lem:HZ} iv). This, together with the spectral sequence in Lemma \ref{lem:A}, implies $H_B^i(\Sc_I(R))_{\mu}=0$ for $i\geqslant1$. The conclusion then follows from i) and ii).
\end{proof}

The following lemma is needed to obtain information about the coefficient module $J$ from information about $R$. Let $\qf$ be a homogeneous prime ideal of $S$ with $\qf\supsetneq B$ and $\kappa(\qf)=S_{\qf}/\qf S_{\qf}$ the residue field of $\qf$. Recall that $U=\Xr\backslash V(I)$ and $\phi|_U:U\rightarrow \PP^{n+1}$ is the morphism restricted to $U$. 

\begin{lem}\label{lem:RIJq}
Let $\mu\in G$ with $H_B^1(\Rc_I(J))_{\mu}=0$. If $\qf$ satisfies one of the following two conditions
 	\begin{itemize}
		\item[i)] $\qf=(H)$, where $H$ is the implicit equation, 
		\item[ii)] $\qf$ is a closed point in the set-theoretic image of $\phi|_U$ with $\dim\phi|_U^{-1}(\qf)=0$ and $\pi_2^{-1}(\qf)\cap \pi_1^{-1}(V(I))=\emptyset$,
	\end{itemize}
    then we have $$(\Rc_I(J)\otimes_S\kappa(\qf))_{\mu}=(\Rc_I(R)\otimes_S\kappa(\qf))_{\mu}.$$
\end{lem}

\begin{proof}
Let $M=R/J\otimes_R \Rc_I(R)\otimes_S \kappa(\qf)$. $\Proj_{\Xr}(M)$ is a closed subscheme of $\Xr\times \PP^{n+1}$ with 
$$\pi_1(\Proj_{\Xr}(M))\subset V(J)\subset V(I)\ \mathrm{and}\ \pi_2(\Proj_{\Xr}(M))\subset V(\qf).$$

For i), Since $\phi|_U$ is generically finite and dominant onto $V(H)$, its generic fiber is finite and contained in $U$. Consequently, the generic fiber of $\Proj_{\Xr}(\Rc_I(R))$ does not intersect the exceptional divisor, i.e., 
\begin{equation}\label{eq:empty}
\pi_2^{-1}(\qf) \cap \pi_1^{-1}(V(I)) = \emptyset.
\end{equation}	
This implies 
$$\Proj_{\Xr}(M)=\emptyset.$$
Since $\Xr$ is smooth at $V(I)$, we deduce $R/J\otimes_R\Rc_I(R)\otimes_S\kappa(\qf)$ is $B$-torsion from \cite[Proposition 5.3.10]{CLS11}. Moreover, from the long exact sequence induced by $0\rightarrow \Rc_I(J)\rightarrow \Rc_I(R) \rightarrow R/J\otimes_R \Rc_I(R)\rightarrow0$ and the assumption on $\mu$, we get $$H_B^0(R/J\otimes_R\Rc_I(R))_{\mu}=0.$$
Since $\qf$ annihilates $R/J\otimes_R\Rc_I(R)$, we have 
\begin{align*}
(R/J\otimes_R\Rc_I(R)\otimes_S\kappa(\qf))_{\mu} &= H_B^0(R/J\otimes_R\Rc_I(R)\otimes_S\kappa(\qf))_{\mu} \\
&= H_B^0(R/J\otimes_R\Rc_I(R)_{\qf})_{\mu} \\
& =0,
\end{align*}
where the last equality follows from $H_B^0(R/J\otimes_R\Rc_I(R))_{\mu}=0$. This proves i).

For ii), we have $\Proj_{\Xr}(M)=\pi_1^{-1}(V(J))\cap \pi_2^{-1}(\qf)$ as sets, which is empty by assumption. Meanwhile, we obtain $H_B^0(R/J\otimes_R\Rc_I(R)\otimes_S\kappa(\qf))_{\mu}=0$ from \cite[Proposition 6.3 (ii)]{Cha13}, which also indicates $(R/J\otimes_R\Rc_I(R)\otimes_S\kappa(\qf))_{\mu}=0$.
\end{proof}

We are now ready to prove the main theorem about quadratic matrix representations.
\begin{thm}\label{thm:MQM}
Let $J$ be a homogeneous ideal satisfying the conditions \ding{173},\ding{174},\ding{175}. Let $\mu\notin \mathscr{A}$ with $H_B^2(J)_{\mu}=H_B^0(H_0)_{\mu+\gamma}=H_B^0(H_0(\mathbf{f};R))_{\mu+\gamma}=0$ and $J_{\mu}\neq 0$. Then the complex
\begin{equation}\label{eq:MQ}
	  \cdots \rightarrow (\Zc_3)_{\mu}\rightarrow (\Zc_2)_{\mu} \rightarrow (\Zc_1)_{\mu}\oplus (Q_2/Q_1)_{\mu}\xrightarrow{\MM_{\mu}} (\Zc_0)_{\mu}
\end{equation}	
is a minimal graded free resolution of the $S$-module $(\Sc_I/H_B^0(\Sc_I))_{\mu}$. Its determinant equals $H^{\deg(\phi)}F$ for some $F\in S$. If $J$ satisfies \ding{176}, then $F\in k^{\times}$.
\end{thm}

\begin{proof}
We deduce $H_B^2(\Zc_0)_{\mu}=H_B^2(J)_{\mu}\otimes_kS=0$ and $H_B^2(K_1)_{\mu+\gamma}=0$ from $H_B^2(J)_{\mu}=0$. Hence, we obtain $H_B^2(Z_1)_{\mu+\gamma}=H_B^0(H_0)_{\mu+\gamma}=0$ by Lemma \ref{lem:HZ} iv). This implies $H_B^1(\Zc_1)_{\mu}=0$. Then we deduce $H_B^0(\Sc_I)_{\mu}=H_B^2(\Zc_2)_{\mu}=H_B^2(Z_2)_{\mu+2\gamma}\otimes_kS(-2)$, which is a free $S$-module generated in degree $2$. Combining with Lemma \ref{lem:A}, this implies that \eqref{eq:MQ} is a minimal graded free resolution of $(\Sc_I/H_B^0(\Sc_I))_{\mu}$. 

We now turn to the determinant. Set $\qf=(H)$. Note that the module $Q(I;J)/Q_1(I,J)$ is supported on the base points. This, together with \eqref{eq:empty}, implies $(\Sc_I/H_B^0(\Sc_I))_{\qf}=(\Rc_I)_{\qf}$. Consequently, we deduce 
\begin{align*}
\dim_{\kappa(\qf)}(\Sc_I/H_B^0(\Sc_I)\otimes\kappa(\qf))_{\mu} &= \dim_{\kappa(\qf)}(\Rc_I(J)\otimes\kappa(\qf))_{\mu} \\
&= \dim_{\kappa(\qf)}(\Rc_I(R)\otimes\kappa(\qf))_{\mu} & \mathrm{Lemma}\  \ref{lem:Rees}\ \mathrm{ii}),\  \ref{lem:RIJq}\ \mathrm{i})   \\
& =\deg (\phi). & \mathrm{Lemma}\ \ref{lem:Rees}\ \mathrm{iii}),\ \mathrm{Equation\  \eqref{eq:H01}}  
\end{align*}
It is well-known that if $R_{\mu}\neq0$, then $\ann_S(\Rc_I(R)_{\mu})=(H)$. Since $\Rc_I(R)$ is an integral domain, we have $\ann_S(\Rc_I(J)_{\mu})=(H)$ whenever $J_{\mu}\neq0$. The computation of the determinant then follows the same lines as the proof of \cite[Theorem 5.5]{Bot11}.
\end{proof}

The rank of the implicitization matrix at a closed point and its relation to multiple points of the parameterization are studied in \cite{BBC14}. This property is used to compute the self-intersection loci of $\mathscr{S}$ in geometric modeling, c.f. \cite{Bus14, JCY22}.

\begin{thm}\label{thm:corank}
Let $\phi$ be generically injective. Under all assumptions of Theorem \ref{thm:MQM}, together with the condition \ding{172} for $J$, if $\qf$ is a closed point in the set-theoretic image of $\phi|_U$ satisfying the following conditions:
 	\begin{itemize}
		\item[i)] $\dim\phi|_U^{-1}(\qf)=0$,
		\item[ii)] $\pi_2^{-1}(\qf)\cap \pi_1^{-1}(V(I))=\emptyset$,
        \item[iii)] $\deg \phi|_U^{-1}(\qf)\leqslant HF(J,\mu)$,
	\end{itemize}
then $\mathrm{corank}(\MM_{\mu}(\qf))=\deg \phi|_U^{-1}(\qf)$.
\end{thm}
\begin{proof}
Denote by $\mathfrak{F}^i_{\mu}(\Sc_I/H_B^0(\Sc_I)\otimes\kappa(\qf))$ the \emph{$i$-th Fitting ideal} of $(\Sc_I/H_B^0(\Sc_I)_{\mu}\otimes\kappa(\qf))_{\mu}$, which is generated by all the minors of size $\HF(J,\mu)-i$ of the matrix $\MM_{\mu}(\qf)$. Then $\mathfrak{F}^i_{\mu}(\Sc_I/H_B^0(\Sc_I)\otimes\kappa(\qf))=0$ is equivalent to $\mathrm{corank}(\MM_{\mu}(\qf))\geqslant i+1$. Furthermore, we have $H_B^i(\Rc_I(R))_{\mu}=0$ for $i\geqslant0$ from Lemma \ref{lem:Rees} iii). Applying \cite[Theorem 6.3 ii)]{Cha13} yields $H_B^0(\Rc_I(R)\otimes\kappa(\qf))_{\mu}=H_B^1(\Rc_I(R)\otimes \kappa(\qf))_{\mu}=0$. Then, we obtain 
\begin{align*}
\dim_{\kappa(\qf)}(\Sc_I/H_B^0(\Sc_I)\otimes\kappa(\qf))_{\mu} &= \dim_{\kappa(\qf)}(\Rc_I(J)\otimes\kappa(\qf))_{\mu}  \\
&= \dim_{\kappa(\qf)}(\Rc_I(R)\otimes\kappa(\qf))_{\mu}  & \mathrm{Lemma}\  \ref{lem:Rees}\ \mathrm{ii}),\  \ref{lem:RIJq}\ \mathrm{ii})   \\
& =\deg \pi_2^{-1}(\qf)   & \ \mathrm{Equation\  \eqref{eq:H01}}          \\
& =\deg \phi|_U^{-1}(\qf).       &  \mathrm{Condition\  ii) }
\end{align*}
Finally, using iii), the conclusion follows from the equivalence of $\mathfrak{F}^i_{\mu}(\Sc_I/H_B^0(\Sc_I)\otimes\kappa(\qf))=0$ and $\dim_{\kappa(\qf)}(\Sc_I/H_B^0(\Sc_I)\otimes\kappa(\qf))_{\mu}\geqslant i+1$, which can be seen by a direct computation of Fitting ideals of vector spaces as in \cite[Proposition 7]{BBC14}.
\end{proof}

 \subsection{Linear matrix representations}
Then we study implicitization using only linear syzygies with respect to a coefficient ideal $J$, that is, implicitizing by the minimal presentation matrix of $(\Sc_I)_{\mu}$.

\medskip
We want $(\Sc_I)_{\mu}=(\Rc_I)_{\mu}$. This requires determining the degrees $\mu$ for which $H^0_{B}(\mathcal{S}_I)_{\mu}=0$. It will be shown that the answer is related to the complement of $\Rf$. Applying Lemma \ref{lem:HJ} iv) recursively for $J=R$ yields 
$$\mathfrak{R}=\bigcup_{p\geqslant2}\left(\bigcup_{0\leqslant j\leqslant p-1}(\Supp_GH_B^p(R)+j\gamma)\right).$$
This form is first introduced in \cite{Bot11}. Set $\cd=\sup\{i \mid H_B^i(R)\neq0 \}$. Compared to $\mathscr{A}$, we note that if $\cd>n$, then $H_{B}^{\mathrm{cd}}(R)_{\mu-(\mathrm{cd}-1)\gamma} \neq 0$ is not allowed for $\mu \notin \Rf$.

\medskip
Based on Lemma \ref{lem:A}, it suffices to consider $H_B^2(\Zc_p)_{\mu}$ for $0\leqslant p\leqslant2$. Given a degree $\mu\in G$. Let $J=\cap_{1\leqslant i\leqslant l} J_i$ be the minimal primary decomposition. We introduce the following condition \eqref{eq:sum}: 
\begin{equation}\label{eq:sum}
H_B^1(R/J)_{\mu-\gamma}=\oplus_i H_B^1(R/J_i)_{\mu-\gamma}.
\tag{$\ast$}
\end{equation}

\begin{lem}\label{lem:H2Z2}
For any $\mu \notin \Rf$ satisfying \eqref{eq:sum}, we have $H_B^2(\Zc_2)_{\mu}=H_B^2(\Zc_1)_{\mu}=H_B^2(\Zc_0)_{\mu}=0$.
\end{lem}
\begin{proof}
$H_B^2(Z_0)_{\mu}=H_B^2(J)_{\mu}$ follows from Lemma \ref{lem:HJ} iv). Moreover, Lemma \ref{lem:HZ} iii) gives $H_B^3(Z_3)_{\mu+2\gamma}=0$. This yields the surjection $H_B^2(K_3)_{\mu+2\gamma}\twoheadrightarrow H_B^2(B_2)_{\mu+2\gamma}$. Moreover, using $H_B^2(H_2)=0$ and Lemma \ref{lem:HJ} ii), we obtain the surjection $\lambda: H_B^2(K_3)_{\mu+2\gamma}\rightarrow H_B^2(Z_2)_{\mu+2\gamma}$. This map is composed of graded components of maps 
$$\varphi:H_B^1(R/J)\rightarrow H_B^2(Z_2).$$
We now analyze the left-hand side using the condition \eqref{eq:sum}. Let $\Sigma(1)=\{\rho_1,\cdots,\rho_l\}$ and $R=k[s_{1},\cdots,s_l]$, where $s_i$ corresponds to $\rho_i$. Choose a cone $\sigma\in \Sigma$ with $\dim \sigma=n$. Let $\mathcal{E}=\{i\mid\rho_i \npreceq \sigma \}$ and $\tau=\prod_{i\in\mathcal{E}} s_i$. Without loss of generality, assume that $\Supp_{\Xr}(J_1)$ is the point defined by $s_i=1$ for $i\in \mathcal{E}$ and $s_i=0$ for $i\notin \mathcal{E}$. For any maximal cone $\sigma'\neq\sigma$, we can deduce that $\rho_i\preceq \sigma'$ for some $i\in \mathcal{E}$.  This implies that all generators of $B$ except $\tau$ become nilpotent modulo $J_1$. Consequently,
$$H_B^1(R/J_1)=H_{(\tau)}^1(R/J_1)=(R/J_1)_{
\tau}/(R/J_1).$$
The first equality follows from changing the ring to $R/J$ and taking the radical of $B$ in $R/J$. The second comes from the Čech complex. Thus, it is an Artinian module divisible by $\tau$. On the other hand, consider the following exact sequence 
$$\xymatrix{0 \ar[r] & H_B^0(H_1) \ar[r]^{d} &  H_B^2(Z_2)  \ar[r]^{\delta_2^2} & H_B^2(K_2)}$$ 
induced by $0\rightarrow Z_2\rightarrow K_2\rightarrow B_1\rightarrow 0$ and $H_B^0(H_1)=H_B^1(B_1)$. Since $I\subset J$, the module $H_B^1(R/J)$ is annihilated by $I$. Hence, the composite map $ \delta_2^2 \cdot \varphi$ is zero by the construction of Koszul complex, which implies $\im(\varphi)\subset \im(d)$. Moreover, $H_B^0(H_1)$ is a Noetherian $B$-torsion module, so $B^k\cdot \im(\varphi)=0$ for some $k>0$. Now take $a\in H_B^1(R/J)_{\mu-\gamma}$ and write $a=\sum a_i$ with $a_i\in H_B^1(R/J_i)_{\mu-\gamma}$. We have 
$$\varphi(a_1)=\tau^k\cdot \varphi\left(\frac{a_1}{\tau^k}\right)=0.$$
Similar results hold for each $a_i$. Hence, we deduce $\varphi=0$, which indicates $H_B^2(Z_2)_{\mu+2\gamma}=0$ by the surjectivity of $\lambda$. Finally, $H_B^2(Z_1)_{\mu+\gamma}=0$ follows in the same way.
\end{proof}

This leads to the main theorem about linear matrices.

\begin{thm}\label{thm:MPM}
Let $J$ be a homogeneous ideal satisfying the conditions \ding{173},\ding{174},\ding{175}. Let $\mu\notin \Rf$ with $(H_0(\mathbf{f};R))_{\mu+\gamma}=0$ and $J_{\mu}\neq 0$. Then the complex 
\begin{equation}\label{eq:MP}
	  \cdots \rightarrow (\Zc_3)_{\mu}\rightarrow (\Zc_2)_{\mu} \rightarrow (\Zc_1)_{\mu}\xrightarrow{\MM_{\mu}} (\Zc_0)_{\mu}
\end{equation}	
is a minimal graded free resolution of the $S$-module $(\Sc_I)_{\mu}$. Its determinant equals $H^{\deg(\phi)}\cdot G$ for some $F\in S$. If $J$ satisfies condition \ding{172}, then $F\in k^{\times}$.
\end{thm}
\begin{proof}
This follows from Lemma \ref{lem:A} ii), Theorem \ref{thm:MQM} and Lemma \ref{lem:H2Z2}.
\end{proof}

\medskip
Compared with Theorem \ref{thm:MQM}, we actually give an upper bound for $H_B^2(J)_{\mu}=H_B^2(H_0)_{\mu+\gamma}=0$. Finally, we provide a criterion for the condition \eqref{eq:sum}.

\begin{prop}\label{prop:ast}
If $\mu\in G$ satisfies $R_{\mu-\gamma}=0$, then $\mu$ satisfies \eqref{eq:sum}.
\end{prop}
\begin{proof}
Let $J=\cap_{1\leqslant i\leqslant l} J_i$ be the minimal primary decomposition. For $2\leqslant j\leqslant l$, $J_j+\cap_{1\leqslant i\leqslant j-1}J_i$ is $B$-torsion. Thus, from the Mayer-Vietoris sequence induced by the primary decomposition, we have the following exact sequence
\begin{equation}\label{eq:Hsum}
0\rightarrow R/(J_j+\cap_{1\leqslant i\leqslant j-1}J_i)\rightarrow H_B^1(R/\cap_{1\leqslant i\leqslant j} J_i) \rightarrow H_B^1(R/\cap_{1\leqslant i\leqslant j-1} J_i)\oplus H_B^1(R/J_j)\rightarrow 0.
\end{equation}
If $R_{\mu-\gamma}=0$, then $(R/(J_j+\cap_{1\leqslant i\leqslant j-1}J_i))_{\mu-\gamma}=0$ for $2\leqslant j\leqslant l$. This implies $H_B^1(R/J)_{\mu-\gamma}=\oplus_i H_B^1(R/J_i)_{\mu-\gamma}$ by induction on $j$.
\end{proof}

\section{Local properties and construction of coefficient ideal}\label{sec:adj}
In this section, we analyze the conditions \ding{172}-\ding{176} locally at the base points and propose several methods to construct the coefficient ideal $J$.

\medskip
 We recall conditions \ding{172}-\ding{176} locally at the base point and introduce a new condition. The smoothness of $\Xr$ at the base points ensures us to study in an $n$-dimensional regular local ring $(R,\mf,k)$. Let $I,J$ be $\mf$-primary ideals. Note that the condition \ding{174} is locally trivial.
\par \ding{172} There exist $f_1'\cdots,f_n'\in I$ such that $IJ=(f_1',\cdots,f_n')J$.
\par \ding{173} $I\subset J$.
\par  \ding{175} There exist $f_1'\cdots,f_m'\in I$ such that $IJ=(f_1',\cdots,f_m')J$ and $f_1',\cdots,f_m'$ form a proper sequence with respect to $J$. 
\par \ding{176} $\Sc_I(J)=\Rc_I(J).$
\par  \ding{177} There exist $f_1'\cdots,f_{n+1}'\in I$ such that $IJ=(f_1',\cdots,f_{n+1}')J$.

\begin{prop}\label{prop:lci}
The following properties hold:
	\begin{itemize}
		\item[i)] \ding{172} $\Rightarrow$ \ding{177}.
		\item[ii)] \ding{172} $\Rightarrow$ \ding{176}.
        \item[iii)] \ding{173} $+$ \ding{177} $\Rightarrow$ \ding{175}.
	\end{itemize}
\end{prop}
\begin{proof}
i) is obvious.

For ii), let $IJ=(f_1',\cdots,f_{n}')J$ and $K=(f_1',\cdots,f_{n}')$. Denote the number of minimal generators of $I$ by $m$. By comparing the depths, $f_1',\cdots,f_n'$ form a regular sequence in $R$, since $IJ$ is $\mf-$primary. It is well-known that $K$ is of linear type. Then we deduce 
$$\langle Q_1(I;J)\rangle=J^{m-n}\oplus (J \cap \langle Q_1(K;R)\rangle )=J^{m-n}\oplus (J \cap Q(K;R))=Q(I;J),$$
where $J^{m-n}$ corresponds to $m-n$ linear relations of $f_i$ with coefficients in $J$. This proves ii). 

 Next, we prove iii). Let $IJ=(f_1',\cdots,f_{n}',g)J$ and $K=(f_1',\cdots,f_{n}')$. By taking generic linear combinations, we may assume $\depth(K)=n$ without loss of generality. This means $Z_1(\mathbf{f'};R)=B_1(\mathbf{f'};R)$. Since $I\subset J$, we have $Z_1(\mathbf{f'};J)=Z_1(\mathbf{f'};R)$. Hence, we obtain
 $$gZ_1(\mathbf{f'};J)=gZ_1(\mathbf{f'};R)\subset J\cdot B_1(\mathbf{f'};R)=B_1(\mathbf{f'};J),$$ which completes the proof.
\end{proof}

\subsection{Adjoint ideals}
We want to find an ideal $J$ that satisfies conditions \ding{172}, \ding{173}. The condition \ding{172} is introduced and named for coefficient ideals in \cite[Definition 2.1]{AH96}. (We do not use their definition. In fact, we seek a $J$ as small as possible with $I\subset J$ because this may yield a smaller matrix.) In some non-l.c.i.\ examples, such coefficient ideals exist in dimension $n\geqslant3$. However, the existence of $J$ for any $I$ is only known in dimension two, which is thoroughly studied in \cite{HS95,Lip94}. We summarize some of their results below. 

Denote the integral closure of $I$ by $\bar{I}$.

\begin{prop}\label{prop:adj} If $n=2$, then the following properties hold:
\par
	\begin{itemize}
		\item[i)] Let $N$ be the number of minimal generators of $\bar{I}$. Then the minimal presentation matrix of $\bar{I}$ is of size $N\times (N-1)$. The ideal generated by all $N-2$ minors of this matrix is equal to the adjoint ideal $\adj(I)$. 
		\item[ii)] The adjoint ideal $\adj(I)$ satisfies conditions \ding{172} and \ding{173}. 
        \item[iii)] If $J$ satisfies \ding{172}, then $J\subset \adj(I)$.
	\end{itemize}
\end{prop}
\begin{proof}
The \emph{adjoint ideal} is first introduced by Joseph Lipman in \cite{Lip94}. Readers unfamiliar with it may treat i) as its definition in dimension two, c.f \cite[Proposition 3.3]{Lip94} and \cite[Corollary 3.6]{HS95}. Condition \ding{172} follows from \cite[Corollary 3.6]{Lip94} and condition \ding{173} from \cite[Remarks 1.2]{Lip94}. Part iii) follows from \cite[Proposition 3.3]{Lip94}. For adjoint ideals in dimension two, see also Section 18.5 in \cite{HS06}.
\end{proof}

\begin{thm}\label{thm:adj}
Use notations of the toric graded setting. A $G$-homogeneous ideal $J$ that satisfies conditions \ding{173}, \ding{174}, \ding{175}, \ding{176} exists in dimension two for any $I$. Moreover, if $k=\CC$ and each $f_i$ is a polynomial over $\QQ$, then $J$ can be generated by polynomials over $\QQ$. 
\end{thm}
\begin{proof} The ideal $J$ can be constructed as follows:
	\begin{itemize}
		\item[i)] Compute all base points $\pf_i$.
		\item[ii)] For each $\pf_i$, compute the adjoint ideals $\adj(I_{\pf_i})$ in the local ring.
        \item[iii)] Homogenize each $\adj(I_{\pf_i})$ in the Cox ring $R$ and obtain a component $J_{i}$ supported on $\pf_i$.
        \item[iv)] Take the intersection $J=\bigcap_i J_i$.
	\end{itemize}
For details on localization and homogenization in the Cox ring, we refer to \cite[Chapter 5]{CLS11}. The constructed ideal $J$ is saturated and supported on $V(I)$ with $J_{\pf_i}=\adj (I_{\pf_i})$. Thus, it satisfies \ding{172},\ding{173} by Proposition \ref{prop:adj}, which indicates \ding{175}, \ding{176} by Proposition \ref{prop:lci}. 

For the second claim, let $R=\CC[s_{1},\cdots,s_l]$ be the Cox ring and let $P=(a_1,\cdots,a_l)$ be a base point with all $a_i\in \CC$. We consider an arbitrary automorphism $\sigma\in \mathrm{Gal}(\QQ[a_1,\cdots,a_l] /\QQ)$. It is easy to see that $\sigma(I)=I$ and $\sigma(P)$ is also a base point. By the commutativity of automorphism and localization, we obtain $\sigma(I_P)=I_{\sigma(P)}.$ Hence, we can construct the coefficient ideal $J$ that coincides at these $\sigma(P)$, which means 
$$\sigma(J_P)=J_{\sigma(P)}.$$
This indicates $\sigma(\bigcap_{\sigma}J_{\sigma(P)})=\bigcap_{\sigma}J_{\sigma(P)}$, which implies that $J=\bigcap_{\sigma}J_{\sigma(P)}$ is generated by polynomials over $\QQ$. 
\end{proof}

In computer algebra systems, only rational numbers can be effectively represented. Theorem \ref{thm:adj} motivates us to seek a computer algorithm for the coefficient ideal over $\QQ$ without the need of computing the base points even when $I$ has irrational base points.

\subsection{Derivative ideals} As mentioned previously, the idea of choosing coefficient ideals originates from \cite{ZSCC03}, where derivative ideals are used. Assume $\mathrm{char}(k)=0$. Set $R=k[s_1,\cdots,s_n]_{(s_1,\cdots,s_n)}$, $I=(f_1,\cdots,f_m)$, with $f_i\in k[s_1,\cdots,s_n]$ for all $i$. The \emph{derivative ideal} is defined as 
$$\partial I=I+\sum_{i,j}(\frac{\partial f_i}{\partial s_j}).$$
It is defined similarly in a polynomial ring. Adjoint ideals are difficult to compute even when the base points are given, whereas derivative ideals are much easier to handle. We prove the validity of this method for integral monomial ideals in dimension two.

\begin{lem}\label{lem:DI}
Let $(R,\mf,k)$ be a two dimensional regular local ring. Let $I$ be a $\mathfrak{m}$-primary ideal and $J$ be the ideal generated by $l$ generic $k$-linear combinations of generators of $I$ with $l\geqslant2$. Then $JI:x=J(I:x)$ for any $x\in \mathfrak{m}\setminus \mathfrak{m}^2$.
\end{lem}
\begin{proof}
Since $x\notin \mathfrak{m}^2$, the quotient $R/(x)$ is a discrete valuation ring of rank one. Hence, there exists $f\in I$ satisfying $(I+(x))/(x)=(f,x)/(x)$. We may then assume $I=(f,xx_1,\cdots,xx_m)$ with $I:x=(x_1,\cdots,x_m)$. Applying generic $k$-linear transformations, we can write $J=(f+xx_1',xx_2'\cdots,xx_l')$, where $x_1',\cdots,x_l'\in(x_1,\cdots,x_m)$. Notice that 
$$fx_i'=(f+xx_1')x_i'-(xx_i')x_1'\in J\cdot(x_1,\cdots,x_m)$$
for $2\leqslant i \leqslant l$, which indicates $$fJ:x=f(J:x)=f(x_2',\cdots,x_l')\subset J\cdot(x_1,\cdots,x_m).$$
Finally, we have 
\begin{align*}
JI:x
& =(fJ+xJ\cdot(x_1,\cdots,x_m)):x\\
& =fJ:x+J\cdot(x_1,\cdots,x_m) \\
& =J\cdot(x_1,\cdots,x_m) \\
& =J(I:x).
\end{align*}
\end{proof}

\begin{thm}\label{thm:DI}
Set $R=k[s_1,s_2]_{(s_1,s_2)}$ with the maximal ideal $\mf=(s_1,s_2)$. let $I$ be a $\mathfrak{m}$-primary monomial ideal and $J$ be an ideal generated by two generic $k$-linear combinations of generators of $I$. Then $I\cdot \partial \bar{I}=J\cdot \partial\bar {I}$.
\end{thm}
\begin{proof}
 For any $x\in\mathfrak{m}\backslash \mathfrak{m}^2$, we have $$\bar{I}(\bar{I}:x)=(\bar{I})^2:x=J\bar{I}:x=J(\bar{I}:x),$$ where the first and third equalities follow from Lemma \ref{lem:DI} and the second follows from \cite[Corollary 13.3.5]{HS06}. For any monomial ideal $K$, we have $\partial K=K:s_1+K:s_2$ since $\mathrm{char} (k)=0$. In addition, it is well-known that the integral closure of a monomial ideal is still a monomial ideal, c.f. \cite[Proposition 1.4.6]{HS06}. Therefore, $$I\cdot \partial\bar{I}=I(\bar{I}:s_1+\bar{I}:s_2)=J(\bar{I}:s_1+\bar{I}:s_2)=J\cdot \partial\bar{I}.$$
\end{proof}
\begin{rmk} From \cite[Theorem 18.4.3]{HS06}, it follows that $\partial \bar{I}\subset \adj(I)$ always holds for any $\mathfrak{m}$-primary monomial ideal $I$. Furthermore, $\partial \bar{I}\varsubsetneq \adj(I)$ if and only if there exists $(n_1,n_2)\in \ZZ^2$ such that $(n_1+1,n_2+1)\in NP^{\circ}(\bar{I})$ while $(n_1+1,n_2),(n_1,n_2+1)\notin NP(\bar{I})$.
\end{rmk}

\medskip
Our results resolve the main conjecture of \cite{ZSCC03}.
\begin{cor}
Let $n=2$. Suppose that all base points are locally of the forms 
$(s_1,s_2)^{n_1}$ or $(s_1^{n_2},s_2^{n_3})$ for some $n_1,n_2,n_3\in \ZZ_{>0}$, which are called $n_1$-ple base point or $n_2\times n_3$-ple base point, respectively.
Then the coefficient ideal can be taken to be $(\partial I)^{sat}$.
\end{cor}
\begin{proof}
The ideal $(\partial I)^{sat}$ satisfies conditions \ding{173} and \ding{174} obviously. All $n_2\times n_3$-ple base points are l.c.i. The condition \ding{172} for $n_3$-ple base points can be derived from 
Theorem \ref{thm:DI} or
$$(s_1,s_2)^{n_1}\cdot (s_1,s_2)^{n_1-1}=(s_1^{n_1},s_2^{n_1})\cdot (s_1,s_2)^{n_1-1}.$$
\end{proof}

\subsection{Toric embeddings induced by monomial supports}
There is some research focusing on implicitization via toric embbedings induced by monomial supports. Our method can be viewed as a generalization of these approaches. Take the following result as an example. 
\begin{prop}{\cite[Theorem 3.8]{BD16}}
Let $\Xr=\mathbb{P}^1\times \mathbb{P}^1$ and $R=k[s_0,s_1;t_0,t_1]$. Assume $I$ has at most l.c.i.\ base points. Suppose that the Newton polygon $NP=NP(f_0,\cdots,f_3)$ is the following region, which is the convex hull of $(\gamma_0,0),(\gamma_1,0),(0,\gamma_2),(\gamma_1,\gamma_2)$ for $\gamma_0,\gamma_1,\gamma_2\in \ZZ_{>0}$ with $\gamma_1>\gamma_0$.
\begin{figure}[htbp]
    \centering
    \includegraphics[width=0.3\linewidth]{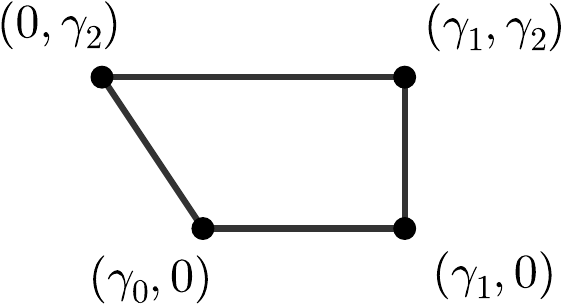}
\end{figure}
\\Then the linear matrix with respect to the monomials supported in the convex hull of $(\gamma_0-1,0),(2\gamma_1-1,0),(0,\gamma_2-1),(2\gamma_1-1,\gamma_2-1)$
is an implicitization matrix.
\end{prop}
\begin{proof}
Firstly, we note that $\pf=(0:1,0:1)$ is a base point of $I$ and $I_{\pf}\subset \langle s_0^{n_1}t_0^{n_2}\mid (n_1,n_2)\in NP  \rangle$. Moreover, we have
$$I_{\pf}\subset J:=\langle s_0^{n_1}t_0^{n_2}\mid n_1,n_2\geqslant0,\frac{n_1}{\gamma_0-1}+\frac{n_2}{\gamma_2-1}\geqslant1  \rangle,$$
which corresponds to the integral points lying above the line segment connecting $(\gamma_0-1,0)$ and $(0,\gamma_2-1)$. The monomials supported in the convex hull of $(\gamma_0-1,0),(2\gamma_1-1,0),(0,\gamma_2-1),(2\gamma_1-1,\gamma_2-1)$ precisely correspond to a basis of $J_{(2\gamma_1-1,\gamma_2-1)}$. Finally, the conclusion follows from Example \ref{ex:region}, Theorem \ref{thm:MPM} and Proposition \ref{prop:ast}.
\end{proof}

In fact, this method performs particularly well for base points at the corner. When $\Xr=\PP^1\times \PP^1$, these points are $(0,1;0,1)$, $(0,1;1,0)$, $(1,0;0,1)$, $(1,0;1,0)$. Beyond this specific case, the general method of coefficient ideals demonstrates flexibility, both in handling base points in arbitrary positions and in allowing for the choice of $J_{\pf}$ that is not necessarily generated by monomials.

\section{Sizes of implicitization matrices}\label{sec:size}
In this section, we study the sizes of the implicitization matrices. A criterion is given for $\MM_{\mu}$ to be square, together with a formula for the number of linear columns. Throughout this section, we assume that $J$ satisfies conditions \ding{173},\ding{174},\ding{175},\ding{176}.

Define $l_{\mu}=\HF(Q_1,\mu),\ q_{\mu}=\HF(Q_2/Q_1,\mu),\   (b_i)_{\mu}=\rank_S((\Zc_{i})_{\mu})\ \text{for}\ i\geqslant2$. The following equation follows from the Euler characteristic of the resolution \eqref{eq:MQ}:
\begin{equation}\label{eq:Euler}
l_{\mu}+q_{\mu}-\HF(J,\mu)=(b_2)_{\mu}+\sum_{i\geqslant3}(-1)^i(b_i)_{\mu}.
\end{equation}

\subsection{Square matrix representations}
It has been extensively studied in \cite{Bot11,BC21} that, in certain specially selected degrees, the free resolution \eqref{eq:MQ} has length at most two. Furthermore, let $f_0',f_1',\cdots,f_n'$ be $n+1$ $k$-linear combinations of $f_0,\cdots,f_{n+1}$. The relations between length-one resolutions, which is equivalent to $\MM_{\mu}$ being square, and the presence of linear syzygies with respect to $f_0',f_1',\cdots,f_n'$ are also studied in \cite{AHW05,BC21,CGZ00}. We prove similar results in our setting.

\begin{prop}\label{prop:Zi}
If $R_{\mu-\gamma}=0$, then the following properties hold:
	\begin{itemize}
		\item[i)] $(\Zc_p)_{\mu}=0$ for $p\geqslant3$. 
		\item[ii)] $\mathcal{Z}_2(\mathbf{f} ;J)_{\mu}\subset \Zc_1(\mathbf{f'};J)_{\mu}$.
 	\end{itemize}	
\end{prop}
\begin{proof} For $p\geqslant3$, we have $Z_p(\mathbf{f};J)=B_p(\mathbf{f};R)$ from Lemma \ref{lem:HH} ii). Moreover, $B_i(\mathbf{f};R)$ is generated by elements in $K_{i+1}(\mathbf{f};R)$. This implies $B_i(\mathbf{f};R)_{\mu+i\gamma}=0$ if $K_{i+1}(\mathbf{f};R)_{\mu+i\gamma}=0$, which is equivalent to $R_{\mu-\gamma}=0.$ This proves i). 

Since $I$ is generated by homogeneous elements, there exists $h\in I_{\gamma}$ such that $I=I'+(h)$. Denote by $H_p'=H_p(\mathbf{f'};J),Z_p'=Z_p(\mathbf{f'};J),B_p'=B_p(\mathbf{f'};J)$. Then we have the following exact sequence (c.f. \cite[Proposition 1.6.12]{BH93}):
\begin{equation}\label{eq:Z2}
 \cdots \longrightarrow H_2'\longrightarrow  H_2 \longrightarrow H_1'(-\gamma) \xrightarrow{\ h\ }  H_1'\longrightarrow \cdots 
\end{equation}
As argued above, $B_2'$, $B_2$ and $B_2(\mathbf{f};R)$ are generated in degree $3\gamma$. Thus, we have ${(B_2')}_{\mu+2\gamma}={(B_2)}_{\mu+2\gamma}={B_2(\mathbf{f};R)}_{\mu+2\gamma}=0$, which indicates ${(H_2')}_{\mu+2\gamma}={(Z_2')}_{\mu+2\gamma}$, ${(H_2)}_{\mu+2\gamma}={(Z_2)}_{\mu+2\gamma}$ and ${(H_2(\mathbf{f};R))}_{\mu+2\gamma}={(Z_2(\mathbf{f};R))}_{\mu+2\gamma}$. Furthermore, we get ${(H_1')}_{\mu+\gamma}={(Z_1')}_{\mu+\gamma}$ for the same reason. Note that $H_2(\mathbf{f'};R)=0$ since $\depth(I',R)=n$. We obtain ${(H_2')}_{\mu+2\gamma}=0$ from $Z_2'\subset Z_2(\mathbf{f}',R)$. Statement ii) now follows from \eqref{eq:Z2}.
\end{proof}

This directly yields the following result on matrix representations. 
\begin{thm}\label{thm:square}
Under the assumptions of Theorem \ref{thm:MQM}, if $R_{\mu-\gamma}=0$, then the free resolution \eqref{eq:MQ} has length at most two. If we further assume $\Syz(\mathbf{f'};J)_{\mu}=0$, then \eqref{eq:MQ} has length one and $\MM_{\mu}$ is square.
\end{thm}

\begin{proof}[Proof for Theorem \ref{thm:det}]
This follows by combining Theorem \ref{thm:MQM}, Theorem \ref{thm:adj} and Theorem \ref{thm:square}.
\end{proof}

If $R_{\mu-\gamma}=0$, then $(b_2)_{\mu}$ is the number of columns minus the number of rows of the quadratic matrix. Let $J_1\subset J_2$ be two ideals. We deduce $\Zc_2(\mathbf{f};J_1)_{\mu}\subset \mathcal{Z}_2(\mathbf{f};J_2)_{\mu}$ from \eqref{eq:Koszul}. This shows that taking coefficient ideals partially achieves our third goal in the following sense: if $\phi$ is l.c.i.\ at the base points, then we can always take $J$ to be some proper ideal of $R$. If $\mu$ still satisfies the required conditions in Theorem \ref{thm:MQM}, then we obtain a matrix representation with fewer rows and a smaller difference between the numbers of columns and rows, compared to the choice $J=R$.

\subsection{Ranks of syzygies}
In the following, we compute the number of linear columns. Moreover, we study the rank of $(Q_k)_{\mu}$ and reveal its relation to $H_B^0(H_0)_{\mu+\gamma}=0$.

Let $(R,\mf,k)$ be a regular local ring. For an $\mf$-primary ideal $I$ in $R$, define $\deg I=\dim_k R/I$ and write $e(I)$ for the Hilbert-Samuel multiplicity. In the graded case, write $\deg I=\sum_{\pf\in V(I)} \deg(I_{\pf})$ and $e(I)=\sum_{\pf\in V(I)} e(I_{\pf})$ when $\dim V(I)=0$. 

\begin{lem}\label{lem:Ik}
Let $I\subset J$ be ideals of a regular local ring $R$, where $I$ is generated by a regular sequence $x_1,\cdots,x_l$. Then $I^k/I^kJ\simeq (R/J)^{\oplus\binom{l+k-1}{k}}$for $k\geqslant1$.  
\end{lem}
\begin{proof}
Consider the minimal free resolution of $I^k$:
$$
\xymatrix{\cdots \ar[r] & R^{N_1} \ar[r]^{d_1} & R^{N_0}  \ar[r]^{d_0} & I^k \ar[r]& 0    }.
$$
It is well-known that the image of $d_1$ is generated by the Koszul syzygies of $x_1,\cdots,x_l$. In fact, the resolution is induced by the Eagon-Northcott complex with $N_0=\binom{l+k-1}{k},N_1=k\binom{l+k-1}{k+1}$, c.f. \cite[Proposition 2.5]{BE75}. By tensoring the complex with $R/J$, we obtain $I^k\otimes_R (R/J)\simeq (R/J)^{\oplus \binom{k+l-1}{k}}$ if and only if $\bar{d_1}=0$, which is equivalent to $I\subset J$.
\end{proof}

\begin{lem}\label{lem:IkJ}
Let $(R,\mf,k)$ be a regular local ring. Let $I\subset J$ be $\mf$-primary ideals of $R$ satisfying the equality in the condition \ding{172} with $n=\dim R$. We have
$$\deg(I^kJ)=\binom{k+n-1}{k}\deg J+\binom{k+n-1}{k-1}e(I).$$
\end{lem}
\begin{proof}
Let $IJ=(f_1',\cdots,f_{n}')J$ and set $K=(f_1',\cdots,f_{n}')$. By comparing the depths, we see that $K$ is a complete intersection and a minimal reduction of $I$. Moreover, it is well-known that $\deg(K)=e(I)$, c.f. \cite[lemma 4.6.5]{BH93}. Hence, by Lemma \ref{lem:Ik}, we obtain
\begin{align*}
\deg(I^kJ) & =\deg(K^kJ) \\
& =\dim_k(K^k/K^kJ)+\deg(K^k)\\
& =\binom{k+n-1}{k}\deg J+\binom{k+n-1}{k-1}e(I).
\end{align*}
\end{proof}

Lemma \ref{lem:IkJ} enables us to compute $\HF(H_0,\mu+\gamma)$ by the Grothendieck-Serre formula. For $\mu\notin \mathscr{A}$ with $H_B^2(J)_{\mu}=0$, we have $H_B^1(H_0)_{\mu+\gamma}=0$ by the spectral sequence in the proof of Lemma \ref{lem:HJ}. Hence,
\begin{align*}
\HF(H_0,\mu+\gamma)+\HF(H_B^0(H_0),\mu+\gamma)& =\HP(H_0,\mu+\gamma)  \\
& =\sum_{\pf\in V(I)}\dim_kJ_{\pf}/I_{\pf}J_{\pf} \\
& =(n-1)\deg J+e(I).
\end{align*}
This shows that $H_B^0(H_0)_{\mu+\gamma}=0$ if and only if 
\begin{equation}\label{eq:H0}
\HF(H_0,\mu+\gamma)=(n-1)\deg J+e(I).
\end{equation}

Now, we are ready to compute the dimensions of the syzygy modules.
\begin{prop}
If $\mu\notin \mathscr{A}$ with $H_B^2(J)_{\mu}=H^0_B(H_0)_{\mu+\gamma}=0$, then $$l_{\mu}=(n+2)\HF(R,\mu) -\HF(R,\mu+\gamma)-2\deg J+e(I).$$
Moreover, we have
\begin{align*}
\HF(Q_k,\mu)& =\binom{n+k+1}{k}\HF(R,\mu)-\HF(R,\mu+k\gamma)\\
&-\frac{k(k+2n+1)}{n(n+1)}\binom{k+n-1}{k}\deg J+\binom{k+n-1}{k-1}e(I).
\end{align*}
$$$$
\end{prop}
\begin{proof}
First,  we deduce $H_B^1(R/J)_{\mu}=H_B^2(J)_{\mu}=0$ from the choice of $\mu$. Hence, $(R/J)_{\mu}=\deg J$ follows from \eqref{eq:H01}. This also shows $\HF(R/J,\mu+k\gamma)=\deg J$ for $k\geqslant1$. The formula for $l_{\mu}$ then follows from the exact sequence
\begin{equation}\label{eq:Q1}
0\rightarrow (Q_1)_{\mu}\rightarrow J_{\mu}^{\oplus(n+2)} \xrightarrow{I} J_{\mu+\gamma} \rightarrow (J/IJ)_{\mu+\gamma}\rightarrow 0.
\end{equation}

From Lemma \ref{lem:HZ} and the proof of Theorem \ref{thm:MQM}, we have $H_B^i(Q_1)_\mu = 0$ for $i \geqslant 0$.
To study $Q_k$, we consider the following exact sequence:
\begin{equation}\label{eq:Qk-exact}
0\rightarrow \ker\varphi \rightarrow \bigoplus_{|\alpha| = k-1}(Q_1)_\mu\mathbf{x}^\alpha \xrightarrow{\varphi} (Q_k)_\mu \rightarrow (Q_k/Q_1)_\mu \rightarrow 0,
\end{equation}
where $|\alpha| = \alpha_0 + \dots + \alpha_{n+1}$ and $\varphi$ maps 
$x^\alpha(a_0x_0 + \dots + a_{n+1}x_{n+1})$ to $a_0(x_0x^\alpha) + \dots + a_{n+1}(x_{n+1}x^\alpha)$.
Note that $\ker\varphi=\ker(\mathcal{Z}_1 \to \mathcal{Z}_0)_k$ by definition. This leads us to compute the local cohomology of $\ker\varphi$.
By the acyclicity of the approximation complex $(\mathcal{Z}_\bullet)_\mu$, we have the exact sequence
\begin{equation}\label{eq:kerphi-res}
\dots \to (\mathcal{Z}_4)_{\mu,k} \to (\mathcal{Z}_3)_{\mu,k} \to (\mathcal{Z}_2)_{\mu,k} \to \ker\varphi \to 0.
\end{equation}
Since $H_B^i(\mathcal{Z}_p)_\mu = 0$ for $i \geqslant 3$, we obtain $H_B^i(\ker\varphi) = 0$ for $i \geq 3$.
This indicates $H_B^i(Q_k)_\mu = 0$ for $i \geqslant 2$ by \eqref{eq:Qk-exact}.
Hence, $H_B^i(J/I^k J)_{\mu + k\gamma} = 0$ for $i \geqslant 0$ follows from the following exact sequence:
\begin{equation}\label{eq:J/IkJ-exact}
0\to (Q_k)_\mu \to J_\mu^{\oplus\binom{n+k+1}{k}} \xrightarrow{I^k} J_{\mu + k\gamma} \to (J/I^k J)_{\mu + k\gamma} \to 0.
\end{equation}
This implies
\[
\mathrm{HF}(J/I^k J, \mu + k\gamma) = \left(\binom{k+n-1}{k} - 1\right)\deg J + \binom{k+n-1}{k-1} e(I)
\]
by the Grothendieck-Serre formula and Lemma \ref{lem:IkJ}. The desired conclusion now follows from computing the Hilbert function of \eqref{eq:J/IkJ-exact}.
\end{proof}

\medskip
In this section, we compute the value of $l_{\mu}$ and give a criterion for $(b_2)_{\mu}=0$. However, we do not propose a general method for choosing $J$ such that $\Syz(\mathbf{f'};J)_{\mu}=0$, which is the condition required in the criterion. In dimension two, the syzygy module of three polynomials has been studied in many cases, including \cite{BDS21,CS03,HW06,SSV14,Wea25}. We hope future research will resolve this issue.

\section{Examples}\label{sec:ex}
In Section \ref{sec:adj}, we propose several ways to construct the coefficient ideal $J$. We will illustrate them with several examples. All computations are completed using \cite{GS}. 

\medskip
First, let us look at two examples of tensor product surfaces, i.e., $\Xr=\PP^1\times \PP^1$. For simplicity, we denote by $R=\mathbb{C}[s_0,s_1;t_0,t_1]$ the homogeneous coordinate ring.
\begin{ex}
\cite[Example 5.7]{BC21}\ 
Consider the rational map $\phi$ defined by the four $3$-minors of the following matrix:
$$\left(\!\begin{array}{ccc}
s_{0}&0&s_{1}^{2}t_{0}^{2}+s_{0}s_{1}t_{1}^{2}\\
s_{1}&t_{0}-t_{1}&s_{0}s_{1}t_{0}^{2}+s_{0}^{2}t_{1}^{2}\\
s_{0}+s_{1}&t_{0}&s_{1}^{2}t_{0}t_{1}+s_{0}s_{1}t_{1}^{2}\\
0&t_{1}&s_{0}^{2}t_{0}t_{1}+s_{1}^{2}t_{1}^{2}
\end{array}\!\right)$$
The bi-degree is $\gamma=(3,3)$. This map has only l.c.i.\ base points with the total multiplicity $\sum_{\pf\in V(I)} e(I_{\pf})=13$. Thus $\phi$ parametrizes a surface with $\deg(H)=5$ by \eqref{eq:DEG1}. According to \cite[Theorem 4.4]{BC21}, the quadratic matrices with respect to $R$ in bi-degree $(1,2)$ and $(2,1)$ both yield the right determinant. They are both of sizes $6\times 8$ with $7$ linear columns and $1$ quadratic column. However, these matrices are not square. 

We can construct a square matrix using our method. The map $\phi$ has two simple base points: 
$$\pf_1=(s_1,t_1),\pf_2=(s_0+s_1,t_1).$$
According to our setting, condition \ding{172} holds automatically in the l.c.i.\  case, so we only need to take care of condition \ding{173}. We take $J_{\pf_1}=\pf_1,J_{\pf_2}=\pf_2$ and set $$J=J_{\pf_1}\cap J_{\pf_2}=(s_0s_1+s_1^{2},\,s_1\,t_1,\,s_0\,t_1+s_1\,t_1,\,t_1^{2}).$$
A $k$-basis of $J_{(2,1)}$ is then computed as
$$\langle s_0s_1t_0+s_1^2t_0,\
 s_0s_1t_1+s_1^2t_1,\ s_1^2t_1,\
  s_0^2t_1+s_0s_1t_1  \rangle.$$
Now, the quadratic matrix with respect to $J$ in bi-degree $(2,1)$ can be computed using the above basis, which is
$$\left(\!\begin{array}{cccc}
0&0&x_{1}-x_{2}&x_{0}x_{3}-x_{2}x_{3}\\
x_{1}+x_{3}&-x_{1}-x_{2}-2\,x_{3}&-x_{0}+x_{2}&2\,x_{0}^{2}-3\,x_{0}x_{2}-x_{2}^{2}-2\,x_{2}x_{3}\\
-x_{2}-x_{3}&2\,x_{2}+2\,x_{3}&0&-x_{0}^{2}-x_{1}^{2}+2\,x_{0}x_{2}+2\,x_{2}^{2}+2\,x_{2}x_{3}\\
0&x_{1}+x_{3}&0&-x_{0}^{2}+x_{0}x_{2}+x_{2}^{2}+x_{2}x_{3}
\end{array}\!\right)$$
with the correct determinant.
\end{ex}

\begin{ex} Consider the rational map $\phi$ defined by the following polynomials:
\begin{align*}
&f_0=s_0^7t_1^7,\\
&f_1=s_0^6s_1t_0t_1^6,\\
&f_2=s_0^2s_1^5t_0^5t_1^2+2s_0^4s_1^3t_0^6t_1,\\
&f_3=s_1^7t_0^7.
\end{align*}
We have $\gamma=(7,7)$ and $\deg(\phi)=3$. There are two non-a.l.c.i.\ base points $\mathfrak{p}_1=(0,1;0,1),\mathfrak{p}_2=(1,0;1,0)$ with 
$$I_{\mathfrak{p}_1}=(s_0^7,s_0^6t_0,s_0^2t_0^5,t_0^7),\ I_{\mathfrak{p}_2}=(s_1^7,s_1^3t_1+s_1^5t_1^2,s_1t_1^6,t_1^7).$$ Due to the presence of non-a.l.c.i.\ base points, Theorem 4.4 of \cite{BC21} cannot be applied to this example. In fact, the linear matrix in bi-degree $(13,6)$ with respect to $R$ is even not of full rank. However, our Theorem \ref{thm:MPM} does provide the right matrix representations. We use the adjoint ideals. We have $\overline{I_{\pf_1}}=(s_0,t_0)^7$ and $\overline{(s_1^7,s_1t_1^6,t_1^7)}=(s_1,t_1)^7$. This implies $\overline{I_{\pf_2}}=\overline{(s_1,t_1)^7+(s_1^3t_1)}=(s_1^7,s_1^3t_1,s_1^2t_1^3,s_1t_1^5,t_1^7)$. Using Proposition \ref{prop:adj} i), set  $$J_{1}=(s_0,t_0)^6,\ J_{2}=(s_1^3,s_1^2t_1,s_1t_1^3,t_1^5),$$
and $J=J_{1}\cap J_{2}$. The linear matrix in bi-degree $(13,6)$ with respect to $J$ is of size $68\times 115$ with the correct determinant. The quadratic matrix in bi-degree $(7,6)$ with respect to $J$ is of size $26\times 32$, containing $31$ linear columns and $1$ quadratic column, and also yields the correct determinant.
\end{ex}

\begin{ex}
Let $\Xr = \mathbb{P}(1,1,2)$ be the weighted projective space with the Cox ring 
$R = \mathbb{C}[s_0, s_1, s_2]$ and the class group $G = \mathbb{Z}$, where $\deg s_0 = \deg s_1 = 1$, $\deg s_2 = 2$. $\Xr$ is singular at $(0:0:1)$. Consider the rational map $\phi : \mathcal{X} \dashrightarrow \mathbb{P}^3$ defined by the 
following polynomials of degree $\gamma = 6$:
\[
\begin{aligned}
f_0 &= s_1^2(s_0^2 - 2s_1^2)^2,\\
f_1 &= (s_0^2+s_1^2)(s_0^2 - 2s_1^2)s_2,\\
f_2 &= s_1^2s_2^2,\\
f_3 &= (s_0^2-2s_1^2)s_2^2+2s_0^2(s_0^2-2s_1^2)^2+3s_2^3.
\end{aligned}
\]
We have $\deg(\phi)=2$. This map has two a.l.c.i.\ base points $\mathfrak{p}_1 = (\sqrt{2}:1:0)$ and 
$\mathfrak{p}_2 = (-\sqrt{2}:1:0)$, both being locally of the form
$$(s^2,st,t^2)$$
with $e(I_{\pf_1})=e(I_{\pf_2})=4$. Thus, the map parametrizes a hypersurface of degree $\frac{1}{2}\times(\frac{1}{2}\times 6^2-8)=5$ by \eqref{eq:DEG1} and \cite[Lemma 6.4.2]{CLS11}. These base points enable us to compute the coefficient ideal by taking derivatives. We set
$$J=(\partial I)^{sat}=(s_0^2-2s_1^2,s_2).$$
Note that $J$ is generated by polynomials over $\QQ$. The linear matrix with respect to $J$ in degree $5$ is
$$\left(\!\begin{array}{cccccccccc}
0&4\,x_{0}-x_{3}&0&0&-x_{1}&-2\,x_{0}-x_{3}&0&0&-4\,x_{1}&0\\
0&3\,x_{2}&0&0&x_{0}&3\,x_{2}&0&0&x_{0}&x_{0}\\
4\,x_{0}-x_{3}&0&-x_{1}&-2\,x_{0}-x_{3}&0&0&-4\,x_{1}&0&0&0\\
3\,x_{2}&0&x_{0}&3\,x_{2}&0&0&x_{0}&x_{0}&0&0\\
0&0&0&0&0&0&0&0&3\,x_{2}&2\,x_{0}\\
0&0&0&0&0&0&3\,x_{2}&2\,x_{0}&0&0\\
0&0&0&0&3\,x_{2}&2\,x_{1}&0&0&6\,x_{2}&-x_{3}\\
0&2\,x_{0}+x_{2}&0&0&0&x_{2}&0&0&0&3\,x_{2}\\
0&0&3\,x_{2}&2\,x_{1}&0&0&6\,x_{2}&-x_{3}&0&0\\
2\,x_{0}+x_{2}&0&0&x_{2}&0&0&0&3\,x_{2}&0&0
\end{array}\!\right)$$
which is square and has the correct determinant.
\end{ex}

Finally, we give a $3$-dimensional example.
\begin{ex}
Consider the rational map from $\mathbb{P}^1\times\mathbb{P}^1\times\mathbb{P}^1$ to $\mathbb{P}^4$ defined by the following polynomials in the homogeneous coordinate ring $\mathbb{C}[s_0,s_1;t_0,t_1;u_0,u_1]$:
\begin{align*}
&f_0=s_0^4 t_1^4 u_1^4+2s_0^4t_0^2t_1^2u_0^4+5s_0^3s_1t_0^3t_1u_1^4+2s_0^4t_0^2t_1^2u_0u_1^3,\\
&f_1=s_1^4 t_0^4 u_1^4,\\
&f_2=s_1^4 t_1^4 u_0^4,\\
&f_3=s_0^4 t_1^4 u_0^4 + s_1^4 t_0^4 u_0^4 + s_0^4 t_0^4 u_0^4,\\
&f_4=s_0^4 t_0^4 u_1^4 + s_0^3 s_1 t_0^3 t_1 u_0^3 u_1 + s_0^4 t_0^4 u_0^4.
\end{align*}
We have $\gamma=(4,4,4)$ and $\deg(\phi)=16$. There is one non-l.c.i.\ base point $\mathfrak{p}=(0,1;0,1;0,1)$ with $I_{\pf}=(s_0^4,t_0^4,u_0^4,s_0^3t_0^3u_0^3)$ and $e(I_{\pf})=64$. By \eqref{eq:DEG1}, the rational map parametrizes a hypersurface of degree $\frac{1}{16}\times(3!\times4^3-64)=20$. We take $$J=(s_0,t_0,u_0)^3$$
since $(s_0^4,t_0^4,u_0^4,s_0^3t_0^3u_0^3)(s_0,t_0,u_0)^3=(s_0^4,t_0^4,u_0^4)(s_0,t_0,u_0)^3$. We choose $\mu$ from Example \ref{ex:region}. The linear matrix in degree $(11,7,3)$ is of size $374\times428$. The quadratic matrix in degree $(7,7,3)$ is of size $246\times 246$ with $172$ linear columns and $74$ quadratic columns. Both matrices yield the correct implicit equation.
\end{ex}

\subsection*{Acknowledgments}
We would like to thank professor Laurent Bus\'e from INRIA for his insightful comments and suggestions for the improvement of the manuscript. This work is supported by the National Natural Science Foundation of China (No. 12494550, 12494555). 

\bibliographystyle{amsalpha}
\def\bibfont{\footnotesize}
\bibliography{bibxmx}

\end{document}